\documentclass[11pt,reqno]{amsart}
\usepackage{hyperref}

\allowdisplaybreaks[4]

\usepackage{amsmath}
\usepackage{amssymb}
\usepackage{mathrsfs}
\usepackage{amsthm}
\usepackage{bm}
\usepackage{bbm}
\usepackage{subfigure}
\usepackage{graphicx}
\usepackage{booktabs}
\usepackage{float}
\usepackage{cases}
\usepackage{geometry}
\usepackage{color}
\theoremstyle{plain}
\newtheorem{lemma}{Lemma}[section]
\newtheorem{thm}[lemma]{Theorem}

\newtheorem{coro}[lemma]{Corollary}
\newtheorem{remark}[lemma]{Remark}
\newtheorem{prop}[lemma]{Proposition}
\newtheorem{assumption}{Assumption}

\newcommand{\rmd}{\mathrm d}
\newcommand{\bbE}{\mathbb E}

\newcommand{\bbH}{\mathbb H}

\newcommand{\bfi}{\mathbf i}

\newcommand{\OO}{\mathcal O}

\newcommand{\HH}{\mathcal H}

\newcommand{\<}{\langle}
\renewcommand{\>}{\rangle}

\begin{document}
\title[An adaptive time-stepping fully discrete scheme]{
An adaptive time-stepping fully discrete scheme for stochastic NLS  equation: Strong convergence and numerical asymptotics
}
\author{Chuchu Chen, Tonghe Dang, Jialin Hong}
\address{LSEC, ICMSEC,  Academy of Mathematics and Systems Science, Chinese Academy of Sciences, Beijing 100190, China,
\and 
School of Mathematical Sciences, University of Chinese Academy of Sciences, Beijing 100049, China}
\email{chenchuchu@lsec.cc.ac.cn; dangth@lsec.cc.ac.cn; hjl@lsec.cc.ac.cn}
\thanks{This work is funded by the National key R\&D Program of China under Grant (No. 2020YFA0713701), National Natural Science Foundation of China (No. 11971470,
No. 11871068, No. 12031020, No. 12022118), and by Youth Innovation Promotion Association CAS}
\begin{abstract}
In this paper, we 
propose and analyze an  adaptive time-stepping fully discrete scheme which possesses  the optimal  strong convergence order 
for the stochastic nonlinear Schr\"odinger equation with multiplicative noise. Based on the splitting skill and the adaptive strategy, the $\bbH^1$-exponential integrability of the numerical solution is obtained, which is a key ingredient to derive the  strong convergence order. We show that the proposed  scheme   converges strongly with orders $\frac12$
in time and $2$ in space.  
To investigate the numerical asymptotic behavior, we establish the large deviation principle for the numerical solution. 
This is the first result on  the study of the large deviation principle for the numerical scheme of stochastic partial differential equations with superlinearly growing drift. 
And as a byproduct, the error of the masses between the numerical and exact solutions is finally obtained. 

\end{abstract}
\maketitle
\section{Introduction} 
 
The 
 stochastic nonlinear Schr\"odinger (NLS)  equation 
is widely used to model the propagation of nonlinear dispersive waves in non-homogeneous or random media, and has important  
applications in various fields such as quantum physics, plasma physics, optical fiber communications  and nonlinear optics (see e.g. \cite{NLSapp, CCCH16, Millet} and references therein). 
In this paper, we focus on the  numerical study 
of the following one-dimensional stochastic NLS equation with multiplicative noise of Stratonovich type
\begin{equation}\label{schrod}
\mathrm{d}u=(\mathbf{i}\Delta u+\mathbf{i}\lambda |u|^{2}u)\mathrm{d}t-\mathbf{i}\sqrt{\epsilon}u\circ \mathrm{d}W(t),\quad \text{in } (0,T]\times \mathcal{O}
\end{equation}
with the initial datum $u(0)=u_0\in L^2(\mathcal{O};\mathbb{C})=:\bbH$ and the homogenous Dirichlet boundary condition,
where $T>0,\;\mathcal{O}=(0,1),$ $\epsilon>0$ denotes the intensity of the noise,  
and $\lambda=1$ or $-1$ corresponds to the focusing or defocusing case, respectively. Here, $\{W(t):t\in[0,T]\}$ is a 
real-valued $Q$-Wiener process on a filtered probability space $(\Omega,\mathcal{F},\{\mathcal{F}_t\}_{t\in[0,T]},\mathbb{P}).$
 There exists an orthonormal basis $\{e_k\}_{k\in\mathbb{N}_+}$ of $L^2(\mathcal{O};\mathbb{R})$ and a sequence of mutually independent, real-valued Brownian motions $\{\beta_k\}_{k\in\mathbb{N}_+}$ such that $W(t)=\sum_{k=1}^{\infty}Q^{\frac{1}{2}}e_k\beta_k(t),\; t\in[0,T]$. 
 
Numerical analysis of stochastic NLS equation \eqref{schrod} has been studied in recent decades, 
 for instance, we refer to  \cite{CCC16} for the $\theta$-scheme, \cite{Debussche} for the Crank--Nicolson scheme, 
 \cite{Cui_2019} for the splitting Crank--Nicolson scheme,  
 \cite{CHW17} for  the modified implicit Euler scheme, and \cite{hong2017} for the multi-symplectic scheme.  These works are  drift-implicit type schemes, while their implementation requires solving an algebraic equation at each iteration step, which needs additional computational effort.  In this regard, it is worth investigating  explicit schemes, which are simple to implement and have lower complexity.  However, the explicit, the exponential and the linear-implicit Euler type schemes with a uniform timestep fail to converge for a stochastic partial differential equation (SPDE) with superlinearly growing drift; see \cite{NLSexp} for the stochastic NLS equation and \cite{beccari} 
 for the parabolic SPDE. To our knowledge, there are only a few works on the convergence analysis of the explicit scheme of the stochastic NLS equation.
  For instance, 
 in \cite{Liujie13}, the author constructs an explicit  splitting scheme in the temporal direction and obtains the convergence order in the probability sense.  
  The author in \cite{NLSexp} proposes a new  kind of explicit splitting  scheme, whose strong convergence order   is $\frac12-$ and $s-$ in the temporal and spatial direction with $Q^{\frac12}\in\mathcal{L}_2^s$ (in this case the solution has $\bbH^s$-regularity),  respectively. 
In order to construct a drift-explicit scheme, whose  strong convergence order is optimal, 
we apply the adaptive timestep skill to adapt the timestep size at each iteration. 
We refer to e.g.  \cite{lord,adapAC} for adaptive  schemes for 
 parabolic SPDEs with non-globally Lipschitz drift.  
To our knowledge, there has been no work on the study of the adaptive time-stepping scheme for the stochastic NLS equation. 
The main purpose of this paper is twofold: \begin{itemize}
\item[(\romannumeral1)] Propose 
a drift-explicit,  adaptive time-stepping fully discrete scheme  for \eqref{schrod}, whose strong convergence order is optimal. 
\item[(\romannumeral2)] Investigate the numerical asymptotic behavior   of the proposed scheme as $\epsilon\to0$ via the large deviation principle (LDP).
\end{itemize}

To be specific, in this work we propose an adaptive time-stepping fully discrete scheme, whose spatial direction is using the spectral Galerkin method, and temporal direction is based on the adaptive  splitting exponential Euler  scheme.
A key ingredient to derive the strong convergence order is the $\bbH^1$-exponential integrability of both the exact and numerical solutions. 
It is studied in \cite{Cui_2019} that the exact solution and the drift-implicit type scheme of the stochastic NLS equation can have this exponential integrability due to the preservation  of the mass of the solutions.
 The author in \cite{NLSexp} 
uses the splitting skill to split the stochastic NLS equation into a Hamiltonian subsystem and a mass-decaying linear  subsystem, so that the exponential integrability of the numerical solution is still possessed. We  remark that this type of exponential integrability also has important applications in  other problems, for instance the large deviation-type result (see e.g \cite[Corollary 3.2]{Cui_2016}).    
 To obtain the  exponential integrability of the drift-explicit,  adaptive time-stepping fully discrete scheme, we combine the splitting skill and the adaptive strategy for the proposed scheme to derive the a.s.-uniform boundedness of the mass of the numerical solution. Based on this $\bbH^1$-exponential integrability and the $\bbH^j\,(j=1,2)$-regularity estimates, it is shown that this fully discrete scheme is convergent with  strong orders 
$\frac{1}{2}$ in time and $2$ in space, which  are optimal in the sense that the orders coincide with the optimal temporal H\"older regularity and spatial Sobolev regularity, respectively.

To further study the asymptotic behavior of the proposed adaptive time-stepping fully discrete scheme, we establish the LDP for the numerical solution. The LDP for the SPDE with small noise is also called the Freidlin--Wentzell LDP, which characterizes the exponential decay probabilities that sample paths of the SPDE deviate from that of the corresponding deterministic equation as the intensity of the noise tends to zero, and has received much attention in recent years (see e.g. \cite{FQ, gautier2005, Gautier_2005}). 
A well-known approach proposed in \cite{Weakbook} to establish the LDP is the weak convergence method, which is by means of the equivalence to the Laplace principle.  
 To apply this
approach, the main difficulty lies in proving the compactness of solutions of the skeleton equation and the stochastic  controlled equation in the infinite-dimensional Banach space $\mathcal{C}([0,T];\bbH_N)$. In this regard, by analyzing the  conditional moment estimation of the solution of the 
stochastic  controlled equation, we prove that the  solution of the proposed fully discrete scheme  satisfies 
the LDP on $\mathcal{C}([0,T];\bbH_N)$ with the rate function given by the corresponding skeleton equation.
To our knowledge, this is the first work on the study of the LDP for the numerical scheme of SPDEs with superlinearly growing drift.
As a byproduct, the error of the masses between the numerical  and exact solutions of \eqref{schrod} is finally obtained.   

The outline of this paper is as follows. In the next section, we propose  the  adaptive time-stepping fully discrete scheme,  and prove the a.s.-uniform boundedness of the mass, the $\bbH^{j}\,(j=1,2)$-regularity estimates and the $\bbH^1$-exponential integrability of the numerical solution. In Section $3$, we derive the  optimal strong convergence order of the fully discrete scheme. Section $4$ is devoted to establishing  the LDP for the solution of the  fully discrete scheme.

 To close this section, we introduce some frequently used notations. The norm and the inner product of $\bbH=L^2(\mathcal{O};\mathbb{C})$ are denoted by $\|\cdot\|$ and $\<u,v\>:=\mathrm{Re}\big[\int_{\mathcal{O}}u(x)\bar{v}(x)\rmd x\big]$, respectively. Denote $L^p(\OO):=L^p(\OO;\mathbb{C}),\,1\leq p\leq \infty,$ $H:=L^2(\mathcal{O};\mathbb{R}).$ Let $H^s:=H^s(\mathcal{O})$ and $\bbH^s:=\bbH^s(\mathcal{O}), \,s\in \mathbb{R}$ denote the real-valued and complex-valued Sobolev spaces, respectively. Then the domain of the Dirichlet  Laplacian operator is $\bbH^1_0\cap\bbH^2.$ We denote the interpolation space of the Dirichlet negative Laplacian operator by $\dot{\bbH}^s,\,s\in\mathbb{R}$. It is known that $\bbH^s$ and $\dot{\bbH}^s$ are equivalent for $s=1,2.$  
 Throughout the paper, we assume that the initial datum $u_0\in\bbH^1_0\cap\bbH^2$ is a deterministic function, and that the operator $Q^{\frac{1}{2}}\in\mathcal{L}^2_2:=\mathcal{L}_2(H;H^2),$ i.e.,
$\|Q^{\frac{1}{2}}\|^2_{\mathcal{L}^2_2}:=\sum_{k=1}^{\infty}\|Q^{\frac{1}{2}}e_k\|^2_{H^2}<\infty.$ 
And hence $\|Q^{\frac{1}{2}}\|_{\mathcal{L}(H;H^2)}\leq \|Q^{\frac{1}{2}}\|_{\mathcal{L}^2_2}<\infty.$
In sequel, $C$ is a constant which may change from one line to another, and sometimes we write $C(a,b,c\ldots)$ to emphasize the dependence on the parameters $a,b,c,\ldots$

 \section{The adaptive time-stepping fully discrete scheme}\label{section2}
 In this section, we first introduce the adaptive time-stepping fully discrete scheme  of \eqref{schrod}. Then we prove the a.s.-uniform boundedness of the mass, the $\bbH^j\,(j=1,2)$-regularity estimates and the $\bbH^1$-exponential integrability of the numerical solution, which are important  in  the estimate of the strong convergence order of this fully discrete scheme. We remark that $\epsilon$ is a fixed positive parameter in this section and the next section, and we do not emphasize the dependence on $\epsilon$ of  solutions of the stochastic NLS equation and its discretizations.
 
 It is known that \eqref{schrod} has the following equivalent It\^o formulation 
\begin{align}\label{Itosense}
\mathrm{d}u=\big(\mathbf{i}\Delta u+\mathbf{i}\lambda |u|^{2}u-\frac{\epsilon}{2}F_Q u\big)\mathrm{d}t-\mathbf{i}\sqrt{\epsilon}u\mathrm{d}W(t),\quad \text{in }(0,T]\times \mathcal{O},
\end{align}
 where $F_Q:=\sum_{k=1}^{\infty}(Q^{\frac{1}{2}}e_k)^2.$ The well-posedness and $\bbH^j\,(j=1,2)$-regularity estimates for \eqref{Itosense} have been studied; see e.g. \cite{CCCH16, Cui_2016,  Debussche2, Debussche, Feng22}. 

It is known that the splitting skill can be used to construct convergent explicit numerical schemes for stochastic NLS equation; see e.g. \cite{Brehier, Liujie13, NLSexp}. 
 Introduce a partition $0=t_0<t_1<\cdots<t_{m}<\cdots<t_{M}=T$ with some $M\in\mathbb{N}_+.$ As is shown in \cite{Cui_2019}, one can split \eqref{Itosense} in the time interval $T_m:=[t_m,t_{m+1})$ into a deterministic NLS equation with random initial datum and a linear SPDE. Precisely, 
 for $t\in T_m,$
\begin{subequations}\label{u^D_tau_m(t)}
\begin{numcases}{}
\rmd u^D_{m}(t)=\bfi \Delta u^D_{m}(t)\rmd t+\bfi \lambda |u^D_{m}(t)|^2u^D_{m}(t)\rmd t,\quad u^D_{m}(t_m)=u^S_{m-1}(t_m),\vspace{1mm}\label{u^D_tau_m(t)_1}\\
\rmd u^S_{m}(t)=-\frac{\epsilon}{2}F_Qu^S_{m}(t)\rmd t-\bfi \sqrt \epsilon u^S_{m}(t)\rmd W(t),\quad u^S_{m}(t_m)=u^D_{m}(t_{m+1}),\label{u^D_tau_m(t)_2}
\end{numcases}
\end{subequations}
especially,  for $t\in T_0,$ the initial datum of \eqref{u^D_tau_m(t)_1} is $u^D_0(0)=u_0.$

Let $N\in\mathbb{N}_+,$ and let $\bbH_N$ be the subspace of $\bbH$ consisting of the first $N$ eigenvectors of the Dirichlet Laplacian operator. Denote by $P^N:\bbH\to \bbH_N$ the spectral Galerkin projection, which is defined by $\<P^Nu,v\>:=\<u,v\>$ for $u\in \bbH,v\in\bbH_N.$ 
Applying the spectral Galerkin method to \eqref{u^D_tau_m(t)} in  the spatial direction, we derive the semi-discrete scheme: For $t\in T_m,$
\begin{subequations}\label{semidiscrete}
\begin{numcases}{}
 \rmd u^{D,N}_{m}(t)=\bfi \Delta u^{D,N}_{m}(t)\rmd t+\bfi \lambda P^N|u^{D,N}_m(t)|^2u^{D,N}_m(t)\rmd t,\quad u^{D,N}_{m}(t_m)=u^{S,N}_{m-1}(t_m), \vspace{1mm}\label{semidiscrete1}\\
\rmd u^{S,N}_{m}(t)=-\frac{\epsilon}{2}P^NF_Qu^{S,N}_{m}(t)\rmd t-\bfi \sqrt \epsilon P^Nu^{S,N}_{m}(t)\rmd W(t),\quad u^{S,N}_{m}(t_m)=u^{D,N}_{m}(t_{m+1}),\label{semidiscrete2}
\end{numcases}
\end{subequations}
where the initial datum is $u^{D,N}_0=P^Nu_0.$ 

To present the adaptive time-stepping scheme, the timestep at each iteration must be adapted with some adaptive timestep function $\tau: \bbH\to\mathbb{R}_+$ to control the numerical solution from divergence. Thus the partition $\{t_m:m=0,\ldots,M\}$ of the split equation \eqref{u^D_tau_m(t)} and the semi-discrete scheme \eqref{semidiscrete} is chosen the same as the one will be used in the fully discrete scheme   \eqref{Itoform}. In this case,  to emphasize the dependence on $T$, we use $M_T$ instead of $M$ in the sequel.
 By further  applying the  adaptive exponential Euler scheme in the temporal direction of \eqref{semidiscrete1}, we obtain the fully discrete scheme, whose differential form reads as: 
\begin{subequations}\label{Itoform}
\begin{numcases}{}
\rmd u^{D,N}_{t,m}=\bfi \Delta u^{D,N}_{t,m}\rmd t+\bfi \lambda S^N(t-t_m)P^N|u^N_m|^2u^N_m\rmd t,\quad u^{D,N}_{t_m,m}=u^N_m,\vspace{1mm}\label{Itoform1}\\
\rmd u^{S,N}_{t,m}=-\frac{\epsilon}{2}P^NF_Qu^{S,N}_{t,m}\rmd t-\bfi \sqrt\epsilon P^Nu^{S,N}_{t,m}\rmd W(t),\quad u^{S,N}_{t_m,m}=u^{D,N}_{t_{m+1},m},\label{Itoform2}
\end{numcases}
\end{subequations}
where $t\in T_m,$  and $t_{m+1}=t_m+\tau_m$ with $\tau_m:=\tau(u^N_m)$. Here,  $S^N(t):=P^Ne^{\bfi t\Delta },$ $ u^N_{m+1}:=u^{S,N}_{t_{m+1},m},$ and the initial datum is $u^N_0=P^Nu_0.$   By \eqref{Itoform1}, we have the explicit one-step scheme for the deterministic part:
\begin{align}\label{onestep}u^{D,N}_{t_{m+1},m}=S^N(\tau_m)(u^N_m+\bfi \lambda |u^N_m|^2u^N_m\tau_m).\end{align} If we denote the flows of $u^{D,N}_{t,m}$ and $u^{S,N}_{t,m}$ by ${\Phi}^{D,N}_{m,t-t_m}$ and ${\Phi}^{S,N}_{m,t-t_m}$, respectively for  $t\in T_m$,   then the solution of the fully discrete scheme \eqref{Itoform} can be expressed as
\begin{align*}
u^N_{m}=u^{S,N}_{t_{m},m-1}=\prod_{j=0}^{m-1}({\Phi}^{S,N}_{j,\tau_j}{\Phi}^{D,N}_{j,\tau_j})u^N_0.
\end{align*}
We remark that if the existing time span is longer than $T$ after adding the last timestep, then we take a smaller timestep such that the existing time span just attains $T$ after adding it. Namely, if $t_{M_T-1}+\tau_{M_T-1}>T$, then we enforce the last timestep $\tau_{M_T-1}:=T-t_{M_T-1}.$ In the sequel, we will give some assumptions on the timestep function so that the numerical solution can attain $T$ with finite many timesteps (see Remark \ref{rmktau_min}). Without loss of generality, we take $1/0=\infty.$

\begin{assumption}\label{assump_step1}
Let $\tau_m$ satisfy 
\begin{align}
&\tau_m\leq \min\big\{L_1\|u^N_m\|^2\|u^N_m\|^{-6}_{L^6(\OO)},T\delta \big\}\qquad a.s.,\label{condition1}\\
&\tau_m\ge (\zeta \|u^N_m\|^{\beta}+\xi)^{-1}\delta\qquad a.s.\label{condition2}
\end{align}
with constants $L_1,\zeta,\beta,\xi>0$ and a small constant $\delta\in(0,1)$.
\end{assumption}

Below, we give the estimate of the mass $\|u^N_m\|^2$ of the solution of \eqref{Itoform}.  
Hereafter, we also use the notation 
$\underline{t}:=\max\{m:t_m\leq t\}$ to represent the maximal timestep number not exceeding $t$. 
\begin{lemma}\label{Hnorm}
Under Assumption \ref{assump_step1} \eqref{condition1},  it holds that
\begin{align*}
\sup_{t\in[0,T]}(\|u^{D,N}_{t,\underline t}\|^2\vee\|u^{S,N}_{t,\underline t}\|^2)\leq e^{L_1T}\|u^N_0\|^2\qquad a.s.
\end{align*}
\end{lemma}
\begin{proof}
By the property $\|S^N(t)\|_{\mathcal{L}(\bbH;\bbH)}=1$ and Assumption \ref{assump_step1} \eqref{condition1}, it follows from \eqref{onestep} that
\begin{align*}
\|u^{D,N}_{t_{m+1},m}\|^2=\|u^N_m+\bfi \lambda|u^N_m|^2u^N_m\tau_m\|^2=\|u^N_m\|^2+\tau^2_m\|u^N_m\|^6_{L^6(\OO)}\leq (1+L_1\tau_m)\|u^N_m\|^2\quad a.s.
\end{align*}
By the It\^o formula,  for $t\in T_m,$
\begin{align}\label{Ito_app_1}
\|u^{S,N}_{t,m}\|^2-\|u^{D,N}_{t_{m+1},m}\|^2&=2\int_{t_m}^t\big\<u^{S,N}_{r,m},-\frac{\epsilon}{2}P^NF_Qu^{S,N}_{r,m}\big\>\rmd r+\epsilon\int_{t_m}^t\sum_{k=1}^{\infty}\|P^Nu^{S,N}_{r,m}Q^{\frac{1}{2}}e_k\|^2\rmd r\notag\\
&=-\epsilon\int_{t_m}^t\sum_{k=1}^{\infty}\|(\mathrm{Id}-P^N)u^{S,N}_{r,m}Q^{\frac{1}{2}}e_k\|^2\rmd r\leq 0\quad a.s.
\end{align}
Hence, combining the above two inequalities gives that 
\begin{align*}
\|u^N_{m+1}\|^2&\leq \|u^{D,N}_{t_{m+1},m}\|^2\leq (1+L_1\tau_m)\|u^N_m\|^2
\leq \prod_{j=0}^m(1+L_1\tau_j)\|u^N_0\|^2\\
&\leq e^{L_1t_{m+1}}\|u^N_0\|^2\leq e^{L_1T}\|u^N_0\|^2\quad a.s.
\end{align*}
Moreover, we derive that for $t\in T_m,$
\begin{align*}
\|u^{D,N}_{t,m}\|^2=\|u^N_m\|^2+\|u^N_m\|^6_{L^6(\OO)}(t-t_m)^2\leq (1+L_1\tau_m)e^{L_1t_m}\|u^N_0\|^2\leq e^{L_1T}\|u^N_0\|^2\quad a.s.
\end{align*}
and
\begin{align*}
\|u^{S,N}_{t,m}\|^2\leq \|u^{D,N}_{t_{m+1},m}\|^2\leq e^{L_1T}\|u^N_0\|^2\quad a.s.
\end{align*}
The proof is finished.
\end{proof}

\begin{remark}\label{rmktau_min}
It follows from Lemma \ref{Hnorm} and Assumption \ref{assump_step1} \eqref{condition2} that
\begin{align*}
\tau_m\ge (\zeta\|u^N_m\|^{\beta}+\xi)^{-1}\delta\ge (\zeta e^{\frac{1}{2}\beta L_1T}\|u^N_0\|^{\beta}+\xi)^{-1}\delta=:\tau_{min}\delta,
\end{align*}
which implies that under Assumption \ref{assump_step1}, the final time $T$ is always attainable, i.e.,
\begin{align*}
M_{T}\leq T(\inf_{t_m\in [0,T]}\tau_m)^{-1}\leq T\tau_{min}^{-1}\delta^{-1}\quad \text{a.s.}
\end{align*}
\end{remark}

\subsection{Regularity analysis}
In this subsection, we give regularity analysis of the solution of the fully discrete scheme, including the $\bbH^j\,(j=1,2)$-regularity estimates and the $\bbH^1$-exponential integrability. 
To this end, we make the following assumption on  adaptive timesteps. Let the Hamiltonian be $\HH(u):=\frac{1}{2}\|\nabla u\|^2-\frac{\lambda}{4}\|u\|^4_{L^4(\OO)},\,u\in\bbH^1$.

\begin{assumption}\label{assump_step2}
Let $\tau_m$ satisfy 
\begin{align}
&\tau_m^{\frac{1}{2}-\gamma}\lambda_N\leq L_2\quad a.s.,\label{condition3}\\
&\tau_m^{\gamma}\HH(u^N_m)\leq L_3\quad a.s.\label{condition4}
\end{align}
for some $\gamma\in(0,\frac{1}{2})$ and constants $L_2,L_3>0,$ where $\lambda_N=N^2\pi^2$ is the $N$-th eigenvalue of the Dirichlet negative Laplacian.
\end{assumption}
\begin{remark}
Note that the Gagliardo--Nirenberg inequality $\|u\|^4_{L^4(\OO)}\leq 2\|u\|^3\|\nabla u\|$, the inverse inequality $\|P^Nu\|_{\dot{\bbH}^s}\leq \lambda_N^{\frac{s}{2}}\|P^Nu\|$  and Lemma \ref{Hnorm} give $\HH(u^N_m)\leq C(\|\nabla u^N_m\|^2+1)\leq C\lambda_N.$ If both $\tau^{\frac{1}{2}-\gamma}_m\lambda_N\leq L_2$ and $\tau^{\gamma}_m\lambda_N\leq L_3$ hold, then Assumption \ref{assump_step2} is satisfied. 
\end{remark}

\begin{prop}\label{H^1regularity}
Under Assumptions \ref{assump_step1} and \ref{assump_step2}, for $p\ge 2,$ there exists a constant $C:=C(p,\epsilon,T,\HH(u^N_0))>0$ such that \begin{align*}\bbE\Big[\sup_{t\in[0,T]}\|u^{D,N}_{t,\underline t}\|^p_{\bbH^1}\Big]+\bbE\Big[\sup_{t\in[0,T]}\|u^{S,N}_{t,\underline t}\|^p_{\bbH^1}\Big]\leq C.
\end{align*}
\end{prop}
\begin{proof}
Direct calculation leads to
\begin{align*}
&D\HH(u)(v)=\<\nabla u,\nabla v\>-\lambda\<|u|^2u,v\>,\\
&D^2\HH(u)(v,w)=\<\nabla w,\nabla v\>-\lambda \<|u|^2v,w\>-2\lambda\<u\mathrm{Re}(\bar{u}v),w\>.
\end{align*}
It follows from the chain rule that
\begin{align*}
&\quad \HH(u^{D,N}_{t_{m+1},m})-\HH(u^N_m)=\int_{t_m}^{t_{m+1}}D\HH(u^{D,N}_{t,m})\rmd u^{D,N}_{t,m}\\&=\int_{t_m}^{t_{m+1}}\Big\<\nabla u^{D,N}_{t,m},\bfi \lambda \nabla (S^N(t-t_m)|u^N_m|^2u^N_m)\Big\>\rmd t -\lambda\int_{t_m}^{t_{m+1}}\Big\<|u^{D,N}_{t,m}|^2u^{D,N}_{t,m},\bfi \Delta u^{D,N}_{t,m}\Big\>\rmd t\\
&\quad -\lambda \int_{t_m}^{t_{m+1}}\Big\<|u^{D,N}_{t,m}|^2u^{D,N}_{t,m},\bfi \lambda S^N(t-t_m)|u^N_m|^2u^N_m\Big\>\rmd t\\
&=\int_{t_m}^{t_{m+1}}\left\<-\Delta u^{D,N}_{t,m},\bfi \lambda (S^N(t-t_m)-\mathrm{Id})P^N|u^N_m|^2u^N_m+\bfi \lambda P^N(|u^N_m|^2u^N_m-|u^{D,N}_{t,m}|^2u^{D,N}_{t,m})\right\>\rmd t\\
&\quad -\int_{t_m}^{t_{m+1}}\Big\<|u^{D,N}_{t,m}|^2u^{D,N}_{t,m},\bfi\lambda^2(S^N(t-t_m)-\mathrm{Id})P^N|u^N_m|^2u^N_m+\bfi \lambda^2P^N(|u^N_m|^2u^N_m-|u^{D,N}_{t,m}|^2u^{D,N}_{t,m})\Big\>\rmd t.
\end{align*}
By properties $\|(S(t)-\mathrm{Id})u\|\leq Ct^{\frac{1}{2}}\|u\|_{\dot{\bbH}^1}$, $\|S(t)\|_{\mathcal{L}(\bbH;\bbH)}=1$ and the Gagliardo--Nirenberg inequality $\|u\|^3_{L^6(\OO)}\leq C\|\nabla u\| \|u\|^2,$ we have that for $t\in T_m,$
\begin{align}\label{holder}
\|u^{D,N}_{t,m}-u^N_m\|&\leq \|(S(t-t_m)-\mathrm{Id})u^N_m\|+\|S(t-t_m)\bfi \lambda |u^N_m|^2u^N_m(t-t_m)\|\notag\\
&\leq (t-t_m)^{\frac{1}{2}}\|u^N_m\|_{\dot{\bbH}^1}+\|u^N_m\|^3_{L^6(\OO)}(t-t_m)\notag\\
&\leq C(t-t_m)^{\frac{1}{2}}\|u^N_m\|_{\dot{\bbH}^1}.
\end{align}
Therefore, combining the cubic difference formula $|u|^2u-|v|^2v=(|u|^2+|v|^2)(u-v)+uv(\overline{u-v})$ and the inverse inequality $\|P^Nu\|_{\dot{\bbH}^s}\leq \lambda_N^{\frac{s}{2}}\|P^Nu\|,$ we obtain that
\begin{align*}
&\quad \HH(u^{D,N}_{t_{m+1},m})-\HH(u^N_m)\\
&\leq \int_{t_m}^{t_{m+1}}\Big[\|\nabla u^{D,N}_{t,m}\|\tau^{\frac{1}{2}}_m\| P^N|u^N_m|^2u^N_m\|_{\dot{\bbH}^2}+\|\nabla u^{D,N}_{t,m}\|\Big\|P^N(|u^N_m|^2+|u^{D,N}_{t,m}|^2)(u^N_m-u^{D,N}_{t,m})\\
&\quad +P^Nu^N_mu^{D,N}_{t,m}(\overline{u^N_m-u^{D,N}_{t,m}})\Big\|_{\dot{\bbH}^1}+\|u^{D,N}_{t,m}\|^3_{L^6(\OO)}\tau^{\frac{1}{2}}_m\|P^N|u^N_m|^2u^N_m\|_{\dot{\bbH}^1}\\
&\quad +\|u^{D,N}_{t,m}\|^3_{L^6(\OO)}\Big\|P^N(|u^N_m|^2+|u^{D,N}_{t,m}|^2)(u^N_m-u^{D,N}_{t,m})+P^Nu^N_mu^{D,N}_{t,m}(\overline{u^N_m-u^{D,N}_{t,m}})\Big\|\Big]\rmd t\\
&\leq C\int_{t_m}^{t_{m+1}}\Big[\|\nabla u^{D,N}_{t,m}\|\tau^{\frac{1}{2}}_m\lambda_N\|u^N_m\|^3_{L^6(\OO)}+\|\nabla u^{D,N}_{t,m}\|\lambda^{\frac{1}{2}}_N(\|u^N_m\|^2_{L^{\infty}(\OO)}+\|u^{D,N}_{t,m}\|^2_{L^{\infty}(\OO)})\|u^N_m-u^{D,N}_{t,m}\|\\
&\quad+\|u^{D,N}_{t,m}\|^3_{L^6(\OO)}\tau^{\frac{1}{2}}_m\lambda^{\frac{1}{2}}_N\|u^N_m\|^3_{L^6(\OO)}+\|u^{D,N}_{t,m}\|^3_{L^6(\OO)}(\|u^N_m\|^2_{L^{\infty}(\OO)}+\|u^{D,N}_{t,m}\|^2_{L^{\infty}(\OO)})\|u^N_m-u^{D,N}_{t,m}\|\Big]\rmd t.
\end{align*}
Applying the Gagliardo--Nirenberg inequalities $\|u\|^3_{L^6(\OO)}\leq C\|\nabla u\| \|u\|^2,\|u\|^2_{L^{\infty}(\OO)}\leq C\|\nabla u\|\|u\|$, the inverse inequality and Lemma \ref{Hnorm} yields \begin{align*}
&\quad \HH(u^{D,N}_{t_{m+1},m})-\HH(u^N_m)\\
&\leq C\int_{t_m}^{t_{m+1}}\Big[\|\nabla u^{D,N}_{t,m}\|\tau^{\frac{1}{2}}_m\lambda_N\|\nabla u^N_m\|\|u^N_m\|^2+\|\nabla u^{D,N}_{t,m}\|\lambda^{\frac{1}{2}}_N\big(\|\nabla u^N_m\|\|u^N_m\|+\|u^N_m-u^{D,N}_{t,m}\|_{\dot{\bbH}^1}\|u^N_m-u^{D,N}_{t,m}\|\big)\times\\
&\quad \|u^N_m-u^{D,N}_{t,m}\|+\|\nabla u^{D,N}_{t,m}\|\|u^{D,N}_{t,m}\|^2\tau^{\frac{1}{2}}_m\lambda^{\frac{1}{2}}_N\|\nabla u^N_m\|\|u^N_m\|^2+\|\nabla u^{D,N}_{t,m}\|\|u^{D,N}_{t,m}\|^2\times\\
&\quad \big(\|\nabla u^N_m\|\|u^N_m\|+\|u^{D,N}_{t,m}-u^N_m\|_{\dot{\bbH}^1}\|u^N_m-u^{D,N}_{t,m}\|\big)\|u^N_m-u^{D,N}_{t,m}\|\Big]\rmd t\\
&\leq C\int_{t_m}^{t_{m+1}}\Big[\|\nabla u^{D,N}_{t,m}\|\|\nabla u^N_m\|\big(\tau^{\frac12}_m\lambda_N+\tau^{\frac12}_m\lambda_N^{\frac12}\big)+\|\nabla u^{D,N}_{t,m}\|\|u^N_m-u^{D,N}_{t,m}\|\times\\
&\quad \Big(\lambda^{\frac12}_N\big(\|\nabla u^N_m\|+\lambda^{\frac12}_N\|u^N_m-u^{D,N}_{t,m}\|^2\big)\Big)\Big]\rmd t.
\end{align*}
Noticing that $\|\nabla u^{D,N}_{t,m}\|\leq \|\nabla u^N_m\|+\|\nabla (u^{D,N}_{t,m}-u^N_m)\|\leq (1+C\tau^{\frac{1}{2}}_m\lambda_N^{\frac{1}{2}})\|\nabla u^N_m\|$ due to \eqref{holder} and the inverse inequality,  we arrive at 
\begin{align*}
&\quad \HH(u^{D,N}_{t_{m+1},m})-\HH(u^N_m)\\
&\leq C\int_{t_m}^{t_{m+1}}\Big((1+\tau^{\frac12}_m\lambda^{\frac12}_N)\tau^{\frac12}_m\lambda_N+(1+\tau^{\frac12}_m\lambda_N^{\frac12})\tau_m^{\frac12}(\lambda_N+\tau_m\lambda_N^{2})\Big)\|\nabla u^N_m\|^2\rmd t\\
&\leq C\int_{t_m}^{t_{m+1}}(\tau^{\frac{1}{2}}_m\lambda_N+\tau_m\lambda_N^{\frac{3}{2}}+\tau_m^{\frac{3}{2}}\lambda^{2}_N+\tau_m^2\lambda_N^{\frac52})\|\nabla u^N_m\|^2\rmd t\\
&\leq CL_2\tau_m^{1+\gamma}\|\nabla u^N_m\|^2
\end{align*}
under the assumption \eqref{condition3}.
Since the Gagliardo--Nirenberg inequality $\|u\|^4_{L^4(\OO)}\leq 2\|\nabla u\|\|u\|^3$ and the Young inequality lead to
\begin{align*}
\HH(u^N_m)\ge \frac{1}{4}(\|\nabla u^N_m\|^2-\|u^N_m\|^6),
\end{align*}
which implies that $\|\nabla u^N_m\|^2\leq C(\HH(u^N_m)+\|u^N_0\|^6),$ we obtain
\begin{align}\label{iterH_D}
\HH(u^{D,N}_{t_{m+1},m})\leq \HH(u^N_m)+C\tau^{1+\gamma}_m(\HH(u^N_m)+1).
\end{align}
Applying the It\^o formula yields
\begin{align*}
&\quad\HH(u^{S,N}_{t,m})-\HH(u^{D,N}_{t_{m+1},m})\\
&=\int_{t_m}^t\big\<\nabla u^{S,N}_{r,m},-\frac{\epsilon}{2}\nabla (F_Qu^{S,N}_{r,m})\big\>\rmd r-\int_{t_m}^t\big\<\nabla u^{S,N}_{r,m},\bfi \sqrt\epsilon u^{S,N}_{r,m}\nabla (\rmd W(r))\big\>\\
&\quad -\lambda \int_{t_m}^t\big\<|u^{S,N}_{r,m}|^2u^{S,N}_{r,m},-\frac{\epsilon}{2}P^NF_Qu^{S,N}_{r,m}\rmd r-\bfi \sqrt \epsilon P^Nu^{S,N}_{r,m}\rmd W(r)\big\>\\
&\quad +\frac{\epsilon}{2}\int_{t_m}^t\sum_{k=1}^{\infty}\|\nabla P^N u^{S,N}_{r,m}Q^{\frac{1}{2}}e_k\|^2\rmd r-\frac{\lambda \epsilon}{2}\int_{t_m}^t\sum_{k=1}^{\infty}\big\<|u^{S,N}_{r,m}|^2(-\bfi P^Nu^{S,N}_{r,m}Q^{\frac{1}{2}}e_k),\\
&\quad -\bfi   P^Nu^{S,N}_{r,m}Q^{\frac{1}{2}}e_k\big\>\rmd r-\lambda \epsilon \int_{t_m}^t\sum_{k=1}^{\infty}\big\<u^{S,N}_{r,m}\mathrm{Re}(\overline{u^{S,N}_{r,m}}(-\bfi P^Nu^{S,N}_{r,m}Q^{\frac{1}{2}}e_k)),-\bfi P^Nu^{S,N}_{r,m}Q^{\frac{1}{2}}e_k\big\>\rmd r.
\end{align*}
Taking the expectation, using inequalities $\|u\|_{L^{\infty}}\leq C\|u\|_{\bbH^1},\;\|u\|^3_{L^6(\OO)}\leq C\|\nabla u\|\|u\|^2,u\in \bbH^1$, $(\|F_Q\|_{L^{\infty}(\mathcal{O})}\vee \|\nabla F_Q\|)\leq \|Q^{\frac{1}{2}}\|^2_{\mathcal{L}^1_2}$ and applying Lemma \ref{Hnorm} lead to 
\begin{align*}
\bbE[\HH(u^{S,N}_{t,m})]-\bbE[\HH(u^{D,N}_{t_{m+1},m})]&\leq C\int_{t_m}^t\|\nabla u^{S,N}_{r,m}\|(\|\nabla F_Q\|\|u^{S,N}_{r,m}\|_{L^{\infty}(\OO)}+\|F_Q\|_{L^{\infty}(\OO)}\|\nabla u^{S,N}_{r,m}\|)\rmd r\\
&\quad +C\int_{t_m}^t\|u^{S,N}_{r,m}\|^3_{L^6(\OO)}\|F_Q\|_{L^{\infty}(\OO)}\|u^{S,N}_{r,m}\|\rmd r+C\int_{t_m}^t\|\nabla u^{S,N}_{r,m}\|^2\|Q^{\frac{1}{2}}\|^2_{\mathcal{L}^1_2}\rmd r\\
&\quad +C\int_{t_m}^t\|u^{S,N}_{r,m}\|^2_{L^{\infty}}\|u^{S,N}_{r,m}\|^2\sum_{k=1}^{\infty}\|Q^{\frac{1}{2}}e_k\|^2_{L^{\infty}(\OO)}\rmd r\\
&\leq C\|Q^{\frac{1}{2}}\|^2_{\mathcal{L}^1_2}\bbE\int_{t_m}^t(\|\nabla u^{S,N}_{r,m}\|^2+1)\rmd r,
\end{align*}
which together with \eqref{iterH_D} and the assumption \eqref{condition4} gives that for $t\in T_m,$
\begin{align*}
\bbE[\HH(u^{S,N}_{t,m})]&\leq \bbE[\HH(u^N_m)]+C\bbE\int_{t_m}^{t_{m+1}}\tau_m^{\gamma}(\HH(u^N_m)+1)\rmd r+C\bbE\int_{t_m}^t(\HH(u^{S,N}_{r,m})+1)\rmd r\\
&\leq \bbE[\HH(u^N_m)]+C\bbE\int_{t_m}^t(\HH(u^{S,N}_{r,m})+1)\rmd r+C\bbE\tau_m.
\end{align*}
By iteration, we have
\begin{align*}
\bbE[\HH(u^{S,N}_{t,\underline{t}})]\leq \bbE[\HH(u^N_0)]+C\int_0^t\bbE[\HH(u^{S,N}_{r,\underline{r}})]\rmd r+CT,
\end{align*}
which implies  
\begin{align*}
\sup_{t\in[0,T]}\bbE[\HH(u^{S,N}_{t,\underline{t}})]\leq (\bbE[\HH(u^N_0)]+C)e^{CT}
\end{align*}
due to the Gr\"onwall inequality.
Hence, one derives $\sup_{t\in[0,T]}(\bbE[\|u^{D,N}_{t,\underline t}\|^2_{\bbH^1}]\vee \bbE[\|u^{S,N}_{t,\underline t}\|^2_{\bbH^1}])\leq C.$ 

Moreover, by utilizing the Burkholder--Davis--Gundy inequality,
we  can also obtain the following supremum type inequality
 \begin{align}\label{BDGsup}
 &\quad\;\bbE\Big[\Big|\sup_{t\in[0,T]}\int_0^t\Big\<\nabla u^{S,N}_{s,\underline{s}},\bfi  \sqrt\epsilon P^N u^{S,N}_{s,\underline{s}} \nabla(\rmd W(s))\Big\>\Big|^2\Big]\notag\\
 &+\bbE\Big[\Big|\sup_{t\in[0,T]}\int_{0}^t\big\<|u^{S,N}_{r,m}|^2u^{S,N}_{r,m},\bfi \sqrt \epsilon P^Nu^{S,N}_{r,m}\rmd W(r)\big\>\Big|^2\Big]\leq C\bbE\Big[\int_0^T\|\nabla u^{S,N}_{s,\underline{s}}\|^2\|Q^{\frac{1}{2}}\|^2_{\mathcal{L}^1_2}\rmd s\Big].
 \end{align} 
Applying the above inequalities, one can  finish the proof for the case of $p=2$. For the case of $p>2,$ it can be proved similarly by means of the It\^o formula, we omit the proof.
\end{proof}

 Below, we prove the $\bbH^1$-exponential integrability for the solution of the fully discrete scheme. To this end, 
  we first present a useful exponential integrability lemma, which is a variant of \cite[Lemma $3.1$]{Cui_2016} or \cite[Lemma $2.1$]{Cui_2019}, and we refer to them for the proofs and more details. 
\begin{lemma}\label{Lemma_exp}
Let $X$ be an $\bbH$-valued adapted stochastic process with continuous sample paths satisfying $\int_{\underline t}^t\|\mu(X_t)\|+\|\sigma(X_t)\|^2\rmd t<\infty\;a.s.\;\forall \,t\in[0,T],$ and $X_t=X_{\underline t}+\int_{\underline t}^t\mu (X_r)\rmd r+\int_{\underline t}^t\sigma(X_r)\mathrm{d}W(r).$ If there are two functionals $V$ and $\overline{V}\in\mathcal{C}^2(\bbH;\mathbb{R})$ and a constant $\alpha>0$ such that $DV(X_s)\mu(X_s)+\frac12 \mathrm{Tr}\big(D^2V(X_s)\sigma(X_s)\sigma(X_s)^*\big)+\frac1{2e^{\alpha(s-\underline t)}}\|\sigma(X_s)^*DV(X_s)\|^2+\overline{V}(X_s)\leq \alpha V(X_s)\;a.s. \;\forall\,s\in[\underline t,t),$ then for $t\in[0,T],$
\begin{align}\label{expo_random}
\bbE\Big[\exp\Big\{\frac{V(X_t)}{e^{\alpha(t-\underline t)}}+\int_{\underline t}^t\frac{\overline V(X_r)}{e^{\alpha(r-\underline t)}}\rmd r\Big\}\Big]\leq \bbE\big[\exp\{V(X_{\underline t})\}\big].
\end{align}
Especially, when $\sigma\equiv 0,$ 
\begin{align}\label{expo_nonrandom}
\exp\Big\{\frac{V(X_t)}{e^{\alpha (t-\underline t)}}+\int_{\underline t}^{t}\frac{\overline{V}(X_r)}{e^{\alpha(r-\underline t)}}\rmd r\Big\}\leq \exp\{V(X_{\underline t})\}\quad a.s.
\end{align} 
\end{lemma}

\begin{prop}\label{expHu^N_m}
Under Assumptions \ref{assump_step1} and \ref{assump_step2}, there exist  constants $\alpha_{\lambda}, C>0$ such that 
$$\sup_{t\in[0,T]}\bbE\Big[\exp\Big\{\frac{\HH(u^{S,N}_{t,\underline t})}{e^{\alpha_{\lambda}t}}\Big\}\Big]\leq C\bbE[\exp\{\HH(u^N_0)\}].$$ \end{prop}
\begin{proof}
Let $\mu_1(u^{D,N}_{t,m})=\bfi \Delta u^{D,N}_{t,m}+\bfi \lambda S^N(t-t_m)|u^N_m|^2u^N_m$ for $t\in T_m.$ Similarly to the proof of Proposition \ref{H^1regularity}, we have
\begin{align*}
D\HH(u^{D,N}_{t,m})\mu_1(u^{D,N}_{t,m})\leq C\tau^{\gamma}_m\|\nabla u^N_m\|^2\leq C\tau^{\gamma}_m(\|\nabla u^{D,N}_{t,m}\|^2+\lambda_N^2\tau_m\|u^N_m\|^2)\leq C_0+C_1\tau^{\gamma}_m\HH(u^{D,N}_{t,m}).
\end{align*}
Applying Lemma \ref{Lemma_exp} \eqref{expo_nonrandom} with $\underline t=t_m,\,\mu=\mu_1,\,\sigma\equiv0,\,V=\mathcal{H},\,\overline{V}=-C_0$ and $\alpha=C_1\tau_m^{\gamma}$, and letting $t=t_{m+1}^-$ and taking the limit, we obtain
\begin{align*}
\exp\Big\{\frac{\HH(u^{D,N}_{t_{m+1},m})}{e^{\alpha \tau_m}}-C_0\int_{t_m}^{t_{m+1}}\frac{1}{e^{\alpha(r-t_m)}}\rmd r\Big\}\leq \exp\{\HH(u^N_m)\}.
\end{align*}
Using the fact that $\int_{t_m}^{t_{m+1}}\frac{1}{e^{\alpha(r-t_m)}}\rmd r=\frac{1-e^{-\alpha \tau_m}}{\alpha}\leq \tau_m$ yields 
\begin{align*}
\exp\Big\{\frac{\HH(u^{D,N}_{t_{m+1},m})}{e^{\alpha \tau_m}}\Big\}\leq \exp\{\HH(u^N_m)+C_0\tau_m\},
\end{align*}
which gives 
\begin{align}\label{expH_determ}
\exp\{\HH(u^{D,N}_{t_{m+1},m})\}\leq \exp\{(\HH(u^N_m)+C_0\tau_m)e^{\alpha \tau_m}\}\leq \exp\{(\HH(u^N_m)+C_0\tau_m)(1+2C_1\tau^{1+\gamma}_m )\}
\end{align}
for $\tau_m\leq T\delta$ with $\delta$ being small.

We claim that \begin{align}\label{claim_Lemma}
\sup_{t\in T_m}\bbE\Bigg[\exp\Big\{\frac{\HH(u^{S,N}_{t,m})}{e^{\alpha_{\lambda}(t-t_m)}}-\int_{t_m}^t\frac{\beta_{\lambda}}{e^{\alpha_{\lambda}(r-t_m)}}
\rmd r\Big\}\Bigg]\leq \bbE[\exp\{\HH(u^{D,N}_{t_{m+1},m})\}],
\end{align}
where $\alpha_{\lambda}=C(e^{3L_1T}\|u^N_0\|^6+1)\|Q^{\frac{1}{2}}\|^2_{\mathcal{L}^2_2}$ and $\beta_{\lambda}=C(e^{6L_1T}\|u^N_0\|^{12}+1)\|Q^{\frac{1}{2}}\|^2_{\mathcal{L}^2_2}.$
 In fact, 
 by letting $\mu_2(u)=-\frac{\epsilon}{2}P^NF_Qu$ and $\sigma_2(u)=-\bfi \sqrt \epsilon P^NuQ^{\frac{1}{2}},$  we obtain
\begin{align*}
&\quad D\HH(u^{S,N}_{t,m})\mu_2(u^{S,N}_{t,m})+\frac{1}{2}\mathrm{Tr}\big[D^2\HH(u^{S,N}_{t,m})\sigma_2(u^{S,N}_{t,m})\sigma_2(u^{S,N}_{t,m})^*\big]+\frac{1}{2e^{\alpha_{\lambda}(t-t_m)}}\|\sigma_2(u^{S,N}_{t,m})^*D\HH(u^{S,N}_{t,m})\|^2\\
&=\epsilon \big\<\nabla u^{S,N}_{t,m},\nabla (-\frac12P^NF_Qu^{S,N}_{t,m})\big\>-\lambda \epsilon \big\<|u^{S,N}_{t,m}|^2u^{S,N}_{t,m},-\frac{1}{2}P^NF_Qu^{S,N}_{t,m}\big\>+\frac{\epsilon}{2}\sum_{k=1}^{\infty}\|\nabla (-\bfi P^Nu^{S,N}_{t,m} Q^{\frac{1}{2}}e_k)\|^2\\
&\quad -\frac{\lambda\epsilon}{2}\sum_{k=1}^{\infty}\big\<|u^{S,N}_{t,m}|^2(P^Nu^{S,N}_{t,m}Q^{\frac{1}{2}}e_k),P^Nu^{S,N}_{t,m}Q^{\frac{1}{2}}e_k\big\>-\lambda\epsilon \sum_{k=1}^{\infty}\big\<u^{S,N}_{t,m}\mathrm{Re}(\overline{u^{S,N}_{t,m}}(-\bfi P^Nu^{S,N}_{t,m}Q^{\frac{1}{2}}e_k)),\\
&\quad-\bfi P^Nu^{S,N}_{t,m}Q^{\frac{1}{2}}e_k\big\>+\frac{\epsilon}{2e^{\alpha_{\lambda}(t-t_m)}}\sum_{k=1}^{\infty}\Big(\big\<\nabla u^{S,N}_{t,m}, -\bfi u^{S,N}_{t,m}(\nabla Q^{\frac{1}{2}}e_k)\big\>-\lambda\big\<|u^{S,N}_{t,m}|^2u^{S,N}_{t,m},-\bfi P^Nu^{S,N}_{t,m}Q^{\frac{1}{2}}e_k\big\>\Big)^2\\
&\leq C\|\nabla u^{S,N}_{t,m}\|^2\|Q^{\frac{1}{2}}\|^2_{\mathcal{L}^1_2}+C\|u^{S,N}_{t,m}\|^3_{L^6(\OO)}\|u^{S,N}_{t,m}\|\|F_Q\|_{L^{\infty}(\OO)}+C\|\nabla u^{S,N}_{t,m}\|^2\|Q^{\frac{1}{2}}\|^2_{\mathcal{L}^1_2}\times\\
&\quad (1+\|u^{S,N}_{t,m}\|^2)
+\frac{C}{2e^{\alpha_{\lambda}(t-t_m)}}(\|\nabla u^{S,N}_{t,m}\|^2\|u^{S,N}_{t,m}\|^2\|Q^{\frac{1}{2}}\|^2_{\mathcal{L}^2_2}+\|u^{S,N}_{t,m}\|^6_{L^6(\OO)}\|u^{S,N}_{t,m}\|^2\|Q^{\frac{1}{2}}\|^2_{\mathcal{L}^1_2})\\
&\leq C\|\nabla u^{S,N}_{t,m}\|^2\|Q^{\frac{1}{2}}\|^2_{\mathcal{L}^2_2}(\|u^{S,N}_{t,m}\|^6+1)+C\|u^{S,N}_{t,m}\|^6\|Q^{\frac{1}{2}}\|^2_{\mathcal{L}^1_2}\\
&\leq C(e^{3L_1T}\|u^N_0\|^6+1)\|Q^{\frac12}\|^2_{\mathcal{L}^2_2}\HH(u^{S,N}_{t,m})+C(e^{6L_1T}\|u^N_0\|^{12}+1)\|Q^{\frac12}\|^2_{\mathcal{L}^2_2},
\end{align*}
where in the second inequality we use the Gagliardo--Nirenberg inequality $\|u\|^3_{L^6(\OO)}\leq C\|\nabla u\|\|u\|^2$ for $u\in \bbH^1$, and in the last step we use the inequality $\|\nabla u^{S,N}_{t,m}\|^2\leq 4\HH(u^{S,N}_{t,m})+\|u^{S,N}_{t,m}\|^6,\,t\in T_m$ and Lemma \ref{Hnorm}.
Applying Lemma \ref{Lemma_exp} \eqref{expo_random} with $\mu=\mu_2,\,\sigma=\sigma_2,\,V=\HH,\,\overline V=-\beta_{\lambda}$ and $\alpha=\alpha_{\lambda}$ leads to \eqref{claim_Lemma}.

Hence, it follows from \eqref{claim_Lemma} and the assumption $\tau_m\leq T\delta$ that
\begin{align*}
&\quad\;\bbE\Big[\exp\Big\{\frac{\HH(u^{S,N}_{t,m})}{e^{\alpha_{\lambda}(t-t_m)}}\Big\}\Big]\leq \bbE\Big[\exp\Big\{\frac{\HH(u^{S,N}_{t,m})}{e^{\alpha_{\lambda}(t-t_m)}}-\int_{t_m}^t\frac{\beta_{\lambda}}{e^{\alpha_{\lambda}(r-t_m)}}
\rmd r\Big\}e^{\beta_{\lambda}(t-t_m)}\Big]\\
&\leq \bbE[\exp\{\HH(u^{D,N}_{t_{m+1},m})\}]e^{T\beta_{\lambda}\delta }\leq \bbE\big[\exp\{(\HH(u^N_m)+C_0\tau_m)(1+2C_1 \tau_m^{1+\gamma})\}\big]e^{T\beta_{\lambda}\delta},
\end{align*}
where in the last step we use \eqref{expH_determ}. 
By considering $e^{-\alpha_{\lambda}t_m}\HH$ instead of $\HH,$ we can obtain
\begin{align*}
\bbE\Big[\exp\Big\{\frac{\HH(u^{S,N}_{t,m})}{e^{\alpha_{\lambda}t}}\Big\}\Big]&\leq \bbE\Big[\exp\Big\{\big(\frac{\HH(u^N_m)}{e^{\alpha_{\lambda}t_m}}+C_0\tau_m\big)(1+2C_1 \tau_m^{1+\gamma})\Big\}\Big]e^{T\beta_{\lambda}\delta}\\
&\leq \bbE\Big[\exp\Big\{\frac{\HH(u^N_m)}{e^{\alpha_{\lambda}t_m}}+2C_1 \tau_m^{1+\gamma}\HH(u^N_m)+C\tau_m\Big\}\Big]e^{T\beta_{\lambda}\delta}\\
&\leq \bbE\Big[\exp\Big\{\frac{\HH(u^N_m)}{e^{\alpha_{\lambda}t_m}}\Big\}\Big]e^{C\delta}
\end{align*}
under the assumption that $\tau^{\gamma}_m\HH(u^N_m)\leq L_3.$
By iteration and using Remark \ref{rmktau_min} lead to
\begin{align*}
\bbE\Big[\exp\Big\{\frac{\HH(u^{S,N}_{t,m})}{e^{\alpha_{\lambda}t}}\Big\}\Big]\leq \bbE[\exp\{\HH(u^N_0)\}]e^{C\delta (m+1)}\leq \bbE[\exp\{\HH(u^N_0)\}]e^{CT\tau_{min}^{-1}}.
\end{align*}
The proof is finished.
\end{proof}

In order to derive the $\bbH^2$-regularity of the solution of the fully discrete scheme, we introduce the  functional 
$
f(u)=\|\Delta u\|^2+\lambda \<\Delta u,|u|^2u\>,\; u\in \bbH^2.
$
\begin{prop}\label{H2regularity}
Under Assumptions \ref{assump_step1} and \ref{assump_step2}, for $p\ge 2$, there exists a constant $C:=C(p,\epsilon,T,f(u^N_0))>0$ such that
\begin{align*}
\bbE\Big[\sup_{t\in[0,T]}\|u^{D,N}_{t,\underline t}\|^p_{\bbH^2}\Big]+\bbE\Big[\sup_{t\in[0,T]}\|u^{S,N}_{t,\underline t}\|^p_{\bbH^2}\Big]\leq C.
\end{align*}
\end{prop}
\begin{proof}
Simple calculations give that
\begin{align*}
Df(u)(v)&=2\<\Delta u,\Delta v\>+2\lambda \<\Delta u,u\mathrm{Re}(\bar{u}v)\>+\lambda \<\Delta u,|u|^2v\>+\lambda \<\Delta v,|u|^2u\>,\\
D^2f(u)(v,w)&=2\<\Delta v,\Delta w\>+2\lambda \<\Delta u,w\mathrm{Re}(\bar{u}v)\>+2\lambda \<\Delta w,u\mathrm{Re}(\bar{u}v)\>+2\lambda \<\Delta u,u\mathrm{Re}(\bar{v}w)\>\\
&\quad +2\lambda \<\Delta u,v\mathrm{Re}(\bar{u}w)\>+\lambda\<\Delta w,|u|^2v\>+2\lambda \<\Delta v,u\mathrm{Re}(\bar{u}w)\>+\lambda\<\Delta v,|u|^2w\>.
\end{align*}
\textit{Step $1$.} By the chain rule, we obtain that for $t\in T_m,$
\begin{align*}
&\quad f(u^{D,N}_{t_{m+1},m})-f(u^N_m)=\int_{t_m}^{t_{m+1}}Df(u^{D,N}_{t,m})\,\rmd u^{D,N}_{t,m}=\int_{t_m}^{t_{m+1}}\Big[2\Big\<\Delta u^{D,N}_{t,m},\bfi \lambda\Delta \big(S^N(t-t_m)|u^N_m|^2u^N_m\big)\Big\> \\&\quad +2\lambda\Big\<\Delta u^{D,N}_{t,m},u^{D,N}_{t,m}\mathrm{Re}\big[\overline{u^{D,N}_{t,m}}\big(\bfi\Delta u^{D,N}_{t,m}+\bfi \lambda S^N(t-t_m)|u^N_m|^2u^N_m\big)\big]\Big\>\\
&\quad+ \lambda \Big\<\Delta u^{D,N}_{t,m},|u^{D,N}_{t,m}|^2\bfi \lambda S^N(t-t_m)|u^N_m|^2u^N_m\Big\>\\
&\quad +\lambda \Big\<|u^{D,N}_{t,m}|^2u^{D,N}_{t,m},\bfi \Delta^2  u^{D,N}_{t,m}+\bfi \lambda \Delta (S^N(t-t_m)|u^N_m|^2u^N_m)\Big\>\Big]\rmd t.
\end{align*}
Utilizing the fact that $2\mathrm{Re}(\bar{u}v)=u\bar{v}+\bar{u}v$ yields 
\begin{align*}
 &\quad f(u^{D,N}_{t_{m+1},m})-f(u^N_m)\\
 &=\int_{t_m}^{t_{m+1}}\Big[2\Big\<\Delta u^{D,N}_{t,m},\bfi \lambda \Delta[(S^N(t-t_m)-\mathrm{Id})|u^N_m|^2u^N_m]\Big\>+2\Big\<\Delta u^{D,N}_{t,m},\bfi \lambda \Delta (|u^N_m|^2u^N_m)\Big\>\\
&\quad +\Big\<\Delta u^{D,N}_{t,m},\bfi \lambda^2 |u^{D,N}_{t,m}|^2S^N(t-t_m)|u^N_m|^2u^N_m\Big\> +\lambda\Big\<\Delta u^{D,N}_{t,m},-\bfi (u^{D,N}_{t,m})^2\Delta \overline{u^{D,N}_{t,m}}\Big\>\\
&\quad+\lambda\Big\<\Delta u^{D,N}_{t,m},-\bfi \lambda (u^{D,N}_{t,m})^2S^N(-(t-t_m))|u^N_m|^2\overline{u^N_m}\Big\> +\lambda \Big\<\Delta u^{D,N}_{t,m},\bfi \lambda |u^{D,N}_{t,m}|^2S^N(t-t_m)|u^N_m|^2u^N_m\Big\>\\
&\quad+\Big\<\Delta u^{D,N}_{t,m},-\bfi \lambda \Delta (|u^{D,N}_{t,m}|^2u^{D,N}_{t,m})\Big\> +\lambda\Big\<\Delta (|u^{D,N}_{t,m}|^2u^{D,N}_{t,m}),\bfi \lambda S^N(t-t_m)|u^N_m|^2u^N_m\Big\>\,\rmd t\\
&=:\int_{t_m}^{t_{m+1}}\sum_{j=1}^8I_j\,\rmd t.
\end{align*}
Noticing that $\Delta (|u|^2u)=2|u|^2\Delta u+4u|\nabla u|^2+2\bar{u}(\nabla u)^2+u^2\Delta \bar{u},$ we arrive at
\begin{align*}
I_2+I_4+I_7&=\Big\<\Delta u^{D,N}_{t,m},\bfi \lambda \Delta (|u^N_m|^2u^N_m-|u^{D,N}_{t,m}|^2u^{D,N}_{t,m})\Big\>+\Big\<\Delta u^{D,N}_{t,m},\bfi \lambda (2|u^N_m|^2\Delta u^N_m+4u^N_m|\nabla u^N_m|^2\\
&\quad +2\overline{u^N_m}(\nabla u^N_m)^2+(u^N_m)^2\Delta \overline{u^N_m})\Big\>-\Big\<\Delta u^{D,N}_{t,m},\bfi \lambda (u^{D,N}_{t,m})^2\Delta \overline{u^{D,N}_{t,m}}\Big\>.
\end{align*}
 It follows from the inverse inequality, the Sobolev embedding inequality $\|u\|_{L^{\infty}(\mathcal{O})}\leq C\|u\|_{\bbH^1},u\in \bbH^1$ and the Young inequality that 
 \begin{align*}
\Big\<\Delta u^{D,N}_{t,m},\bfi\lambda \Delta(|u^N_m|^2u^N_m-|u^{D,N}_{t,m}|^2u^{D,N}_{t,m})\Big\>&\leq C\|\Delta u^{D,N}_{t,m}\|\lambda_N(\|u^N_m\|^2_{L^{\infty}(\OO)}+\|u^{D,N}_{t,m}\|^2_{L^{\infty}(\OO)})\|u^N_m-u^{D,N}_{t,m}\|\\
&\leq C\|\Delta u^{D,N}_{t,m}\|\lambda_N\tau^{\frac{1}{2}}_m(\|u^N_m\|^2_{{\bbH}^1}+\|u^{D,N}_{t,m}\|^2_{\bbH^1})\|u^N_m\|_{\dot{\bbH}^1}\\
&\leq C(\|\Delta u^{D,N}_{t,m}\|^2+\|u^N_m\|^6_{\bbH^1}+\|u^{D,N}_{t,m}\|^6_{\bbH^1}),
\end{align*}
where we have used the assumption \eqref{condition3} so that $\lambda _N\tau^{\frac{1}{2}}_m<\infty.$
Similar techniques, combining the fact that $\<u,\bfi|v|^2u\>=0$ give 
\begin{align*}
&\quad 2\Big\<\Delta u^{D,N}_{t,m},\bfi \lambda |u^N_m|^2\Delta u^N_m\Big\>=2\Big\<\Delta u^{D,N}_{t,m},\bfi\lambda |u^N_m|^2(\Delta u^N_m-\Delta u^{D,N}_{t,m})\Big\>\\
&\leq C\|\Delta u^{D,N}_{t,m}\|\|u^N_m\|^2_{L^{\infty}(\OO)}\lambda_N\|u^N_m-u^{D,N}_{t,m}\|\leq C(\|\Delta u^{D,N}_{t,m}\|^2+\|u^N_m\|^6_{{\bbH}^1}).
\end{align*}
And it can be shown that
\begin{align*}
&\quad \Big\<\Delta u^{D,N}_{t,m},\bfi \lambda \big((u^N_m)^2\Delta \overline{u^N_m}-(u^{D,N}_{t,m})^2\Delta \overline{u^{D,N}_{t,m}}\big)\Big\>\\
&=\Big\<\Delta u^{D,N}_{t,m},\bfi \lambda \big((u^N_m)^2-(u^{D,N}_{t,m})^2\big)\Delta \overline{u^N_m}\Big\>+\Big\<\Delta u^{D,N}_{t,m},\bfi\lambda (u^{D,N}_{t,m})^2\Delta\big(\overline{u^N_m-u^{D,N}_{t,m}}\big)\Big\>\\
&\leq \|\Delta u^{D,N}_{t,m}\|\|u^N_m-u^{D,N}_{t,m}\|(\|u^N_m\|_{\bbH^1}+\|u^{D,N}_{t,m}\|_{\bbH^1})\|\Delta u^{D,N}_{t,m}\|_{L^{\infty}(\OO)}+\|\Delta u^{D,N}_{t,m}\|\|u^{D,N}_{t,m}\|^2_{\bbH^1}\|\Delta (u^N_m-u^{D,N}_{t,m})\|\\
&\leq C\|\Delta u^{D,N}_{t,m}\|\lambda_N\tau^{\frac{1}{2}}_m\big[\|u^N_m\|_{\bbH^1}\|u^{D,N}_{t,m}\|_{\bbH^1}(\|u^N_m\|_{\bbH^1}+\|u^{D,N}_{t,m}\|_{\bbH^1})+\|u^N_m\|_{\bbH^1}\|u^{D,N}_{t,m}\|^2_{\bbH^1}\big]\\
&\leq C(\|\Delta u^{D,N}_{t,m}\|^2+\|u^N_m\|^6_{\bbH^1}+\|u^{D,N}_{t,m}\|^6_{\bbH^1})
\end{align*}
and
\begin{align*}
&\quad \Big\<\Delta u^{D,N}_{t,m},\bfi\lambda (4u^N_m|\nabla u^N_m|^2+2\overline{u^N_m}(\nabla u^N_m)^2)\Big\>\leq C\|\Delta u^{D,N}_{t,m}\|\|u^N_m\|_{\bbH^1}\|\nabla u^N_m\|^2_{L^4(\OO)}\\
&\leq C\|\Delta u^{D,N}_{t,m}\|\|u^N_m\|_{\bbH^1}\|\Delta u^N_m\|^{\frac{1}{2}}\|\nabla u^N_m\|^{\frac{3}{2}}\leq C\|\Delta u^{D,N}_{t,m}\|\|u^N_m\|^{\frac{5}{2}}_{\bbH^1}\big(\|\Delta u^{D,N}_{t,m}\|^{\frac{1}{2}}+\|\Delta (u^{D,N}_{t,m}-u^N_m)\|^{\frac{1}{2}}\big)\\
&\leq C(\|\Delta u^{D,N}_{t,m}\|^2+\|u^N_m\|^{10}_{\bbH^1}+1).
\end{align*}
Moreover, the remaining terms can be estimated as follows:
\begin{align*}
I_1\leq C\|\Delta u^{D,N}_{t,m}\|\lambda_N\tau^{\frac{1}{2}}_m\|u^N_m\|^3_{\bbH^1}\leq C\|\Delta u^{D,N}_{t,m}\|^2+C\|u^N_m\|^6_{\bbH^1},
\end{align*}
\begin{align*}
I_3+I_5+I_6\leq C\|\Delta u^{D,N}_{t,m}\|\|u^{D,N}_{t,m}\|^2_{\bbH^1}\|u^N_m\|^3_{L^6(\OO)}\leq C\|\Delta u^{D,N}_{t,m}\|^2+C\|u^N_m\|^6_{\bbH^1}+C\|u^{D,N}_{t,m}\|^6_{\bbH^1}
\end{align*}
and
\begin{align*}
I_8=-\lambda\Big\<\nabla (|u^{D,N}_{t,m}|^2u^{D,N}_{t,m}),\bfi\lambda S^N(t-t_m)\nabla (|u^N_m|^2u^N_m)\Big\>\leq C\|u^{D,N}_{t,m}\|^3_{\bbH^1}\|u^N_m\|^3_{\bbH^1}\leq C\|u^{D,N}_{t,m}\|^6_{\bbH^1}+C\|u^N_m\|^6_{\bbH^1},
\end{align*}
where we use the integration by parts formula and the fact that $\bbH^s$ is an algebra for $s> \frac{1}{2}$, i.e., $\|uv\|_{\bbH^s}\leq C\|u\|_{\bbH^s}\|v\|_{\bbH^s}$ for $u,v\in\bbH^s.$

Combining terms $I_i,i=1,\ldots,8,$ we derive
\begin{align*}
f(u^{D,N}_{t_{m+1},m})-f(u^N_m)\leq C\int_{t_m}^{t_{m+1}}(\|\Delta u^{D,N}_{t,m}\|^2+\|u^N_m\|^{10}_{\bbH^1}+\|u^{D,N}_{t,m}\|^6_{\bbH^1}+1)\rmd t.
\end{align*}
Since the Gagliardo--Nirenberg inequality and the Young inequality give 
\begin{align}\label{fgeDeltau}
f(u)\ge \|\Delta u\|^2-\|\Delta u\|\|u\|^3_{L^6(\OO)}\ge \frac{1}{2}(\|\Delta u\|^2-\|u\|^6_{L^6(\OO)})\ge \frac{1}{2}\|\Delta u\|^2-C\|u\|^2_{\bbH^1}\|u\|^4,
\end{align}
we obtain
\begin{align*}
f(u^{D,N}_{t_{m+1},m})-f(u^N_m)\leq C\int_{t_m}^{t_{m+1}}(f(u^{D,N}_{t,m})+\|u^N_m\|^{10}_{\bbH^1}+\|u^{D,N}_{t,m}\|^6_{\bbH^1}+1)\rmd t,
\end{align*}
which implies
\begin{align}\label{f(u^{D,N}_{m+1})}
f(u^{D,N}_{t_{m+1},m})\leq \Big(f(u^N_m)+C\int_{t_m}^{t_{m+1}}(\|u^N_m\|^{10}_{\bbH^1}+\|u^{D,N}_{t,m}\|^6_{\bbH^1}+1)\rmd t \Big)e^{C\tau_m}.
\end{align}
\textit{Step $2$.} Applying  the It\^o formula yields
\begin{align*}
 & f(u^{S,N}_{t,m})-f(u^{D,N}_{t_{m+1},m})=\int_{t_m}^t2\Big\<\Delta u^{S,N}_{r,m},\Delta (-\frac{\epsilon}{2}P^NF_Qu^{S,N}_{r,m}\rmd r-\bfi \sqrt \epsilon P^Nu^{S,N}_{r,m}\rmd W(r))\Big\>\\
  &+2\lambda \int_{t_m}^t\Big\<\Delta u^{S,N}_{r,m},u^{S,N}_{r,m}\mathrm{Re}(\overline {u^{S,N}_{r,m}}(-\frac{\epsilon}{2}P^NF_Qu^{S,N}_{r,m}\rmd r-\bfi \sqrt \epsilon P^Nu^{S,N}_{r,m}\rmd W(r)))\Big\>\\
 &+\lambda\int_{t_m}^t\Big\<\Delta u^{S,N}_{r,m},|u^{S,N}_{r,m}|^2(-\frac{\epsilon}{2}P^NF_Qu^{S,N}_{r,m}\rmd r-\bfi \sqrt \epsilon P^Nu^{S,N}_{r,m}\rmd W(r))\Big\>\\
 &+\lambda \int_{t_m}^t\Big\<|u^{S,N}_{r,m}|^2u^{S,N}_{r,m},\Delta (-\frac{\epsilon}{2}P^NF_Qu^{S,N}_{r,m}\rmd r-\bfi \sqrt \epsilon P^Nu^{S,N}_{r,m}\rmd W(r))\Big\>\\
 &+\lambda \epsilon \sum_{k=1}^{\infty}\int_{t_m}^t\Big\<\Delta P^N(\bfi u^{S,N}_{r,m}Q^{\frac{1}{2}}e_k),\bfi |u^{S,N}_{r,m}|^2P^Nu^{S,N}_{r,m}Q^{\frac{1}{2}}e_k\Big\>\rmd r\\
 &+\lambda \epsilon \int_{t_m}^t\sum_{k=1}^{\infty}\Big\<\Delta u^{S,N}_{r,m},u^{S,N}_{r,m}|P^Nu^{S,N}_{r,m}Q^{\frac{1}{2}}e_k|^2\Big\>\rmd r+\epsilon \int_{t_m}^t\sum_{k=1}^{\infty}\|\Delta P^N(u^{S,N}_{r,m}Q^{\frac{1}{2}}e_k)\|^2\rmd r\\
 &+2\lambda \epsilon \int_{t_m}^t\sum_{k=1}^{\infty}\Big[\Big\<\Delta u^{S,N}_{r,m},(-\bfi P^Nu^{S,N}_{r,m}Q^{\frac{1}{2}}e_k)\mathrm{Re}(\overline {u^{S,N}_{r,m}}(-\bfi P^Nu^{S,N}_{r,m}Q^{\frac{1}{2}}e_k))\Big>\\
 &+\epsilon\Big\<\Delta (-\bfi P^Nu^{S,N}_{r,m}Q^{\frac{1}{2}}e_k),u^{S,N}_{r,m}\mathrm{Re}\big(\overline {u^{S,N}_{r,m}}(-\bfi P^Nu^{S,N}_{r,m}Q^{\frac{1}{2}}e_k)\big)\Big\>\Big]\rmd r.
\end{align*}
By taking expectation, combining \eqref{f(u^{D,N}_{m+1})} and the fact that $e^{C\tau_m}\leq 1+2C\tau_m$ for $\tau_m\leq T\delta$ with $\delta$ being small, it follows from the fact that $\bbH^s$ is an algebra for $s> \frac{1}{2}$ that
\begin{align*}
\bbE[f(u^{S,N}_{t,m})]&\leq \bbE[f(u^N_m)(1+2C\tau_m)]+C\bbE\Big[(1+2C\tau_m)\Big(\int_{t_m}^{t_{m+1}}(\|u^N_m\|^{10}_{\bbH^1}+\|u^{D,N}_{r,m}\|^6_{\bbH^1}+1)\rmd r\Big)\Big]\\
&\quad +C\bbE\Big[\int_{t_m}^t(\|\Delta u^{S,N}_{r,m}\|^2+\|u^{S,N}_{r,m}\|^6_{\bbH^1}+1)\rmd r\Big]\\
&\leq \bbE[f(u^N_m)(1+2C\tau_m)]+C\bbE\Big[(1+2C\tau_m)\Big(\int_{t_m}^{t_{m+1}}(\|u^N_m\|^{10}_{\bbH^1}+\|u^{D,N}_{r,m}\|^6_{\bbH^1}+1)\rmd r\Big)\Big]\\
&\quad +C\bbE\Big[\int_{t_m}^t(f(u^{S,N}_{r,m})+\|u^{S,N}_{r,m}\|^6_{\bbH^1}+1)\rmd r\Big],
\end{align*}
where we use \eqref{fgeDeltau} in the last step.
By iteration, we derive
\begin{align*}
\bbE[f(u^{S,N}_{t,m})]&\leq \bbE[f(u^N_0)]+C\bbE\Big[\sum_{k=0}^mf(u^N_k)\tau_k\Big]+C\bbE\Big[\sum_{k=0}^{m-1}\int_{t_k}^{t_{k+1}}f(u^{S,N}_{r,k})\rmd r\Big]+C\bbE\Big[\int_{t_m}^tf(u^{S,N}_{r,m})\rmd r\Big]\\
&\quad +C\bbE\Big[\sum_{k=0}^{m}\int_{t_k}^{t_{k+1}}(\|u^N_k\|^{10}_{\bbH^1}+\|u^{D,N}_{r,k}\|^6_{\bbH^1}+1)\rmd r+\sum_{k=0}^{m-1}\int_{t_k}^{t_{k+1}}(\|u^{S,N}_{r,k}\|^6_{\bbH^1}+1)\rmd r\Big]\\
&\quad +C\bbE\Big[\int_{t_m}^t(\|u^{S,N}_{r,m}\|^6_{\bbH^1}+1)\rmd r\Big].
\end{align*}
We claim that for $t\in T_m,$ 
\begin{align}\label{Eu^NtoEu^SN}\bbE\Big[\sum_{k=0}^mf(u^N_k)\tau_k\Big]\leq C\Big(\int_0^t\bbE[f(u^{S,N}_{r,\underline r})]\rmd r+1\Big).
\end{align}
 In fact, 
noticing that for $k=1,2,\ldots,m,$
\begin{align*}
\|u^N_k\|^2_{\bbH^2}\tau_{k-1}\leq 2\int_{t_{k-1}}^{t_k}\|u^{S,N}_{r,k-1}\|^2_{\bbH^2}\rmd r+2\int_{t_{k-1}}^{t_k}\|u^{S,N}_{r,k-1}-u^{S,N}_{t_k,k-1}\|^2_{\bbH^2}\rmd r 
\end{align*}
and $\frac{\tau_k}{\tau_{k-1}}\leq \frac{T\delta}{\tau_{min}\delta}\leq C,$
 we obtain
 \begin{align}\label{u^Ntou^SN}
 &\quad \sum_{k=0}^m\|u^N_k\|^2_{\bbH^2}\tau_k\leq \|u^N_0\|^2_{\bbH^2}\tau_0+C\sum_{k=1}^{m}\|u^N_k\|^2_{\bbH^2}\tau_{k-1}\notag\\
 &\leq \|u^N_0\|^2_{\bbH^2}\tau_0+C\sum_{k=1}^m\Big[\int_{t_{k-1}}^{t_{k}}\|u^{S,N}_{r,k-1}\|^2_{\bbH^2}\rmd r+\int_{t_{k-1}}^{t_k}\|u^{S,N}_{r,k-1}-u^{S,N}_{t_k,k-1}\|^2_{\bbH^2}\rmd r\Big].
 \end{align}
 Utilizing the property of the conditional expectation and the fact that $t_k=t_{k-1}+\tau_{k-1}$ is $\mathcal{F}_{t_{k-1}}$-measurable, one arrives at that for $r\in T_{k-1},$
 \begin{align*}
 \bbE\Big[\big\|\int_{r}^{t_k}P^Nu^{S,N}_{s,k-1}\rmd W(s)\big\|^2\Big]&=\bbE\Big[\bbE\Big[\big\|\int_{r}^{t_k}P^Nu^{S,N}_{s,k-1}\rmd W(s)\big\|^2\Big|\mathcal{F}_{t_{k-1}}\Big]\Big]\\
 &= \bbE\Big[\bbE\Big[\|\int_r^{y}P^N\widehat{u^{S,N}_{s,k-1}}\rmd W(s)\|^2\Big]\Big|_{y=t_k,z_0=u^{S,N}_{t_{k-1},k-1}}\Big]\\
 &\leq \bbE\Big[\bbE\Big[\int_{r}^{y}\sum_{j=1}^{\infty}\|P^N\widehat{u^{S,N}_{s,k-1}}Q^{\frac{1}{2}}e_j\|^2\rmd s\Big]\Big|_{y=t_k,z_0=u^{S,N}_{t_{k-1},k-1}}\Big]\\
 &\leq C\bbE [\tau_{k-1}],
 \end{align*}
 where $\widehat{u^{S,N}_{s,k-1}},s\in T_{k-1}$ is the solution of \eqref{Itoform2} with initial datum $z_0$ at $t_{k-1}$. 
The above inequality  yields 
 \begin{align*}
 \bbE[\|u^{S,N}_{r,k-1}-u^{S,N}_{t_k,k-1}\|^2]\leq C\bbE\Big[\big\|\int_{r}^{t_k}P^NF_Qu^{S,N}_{s,k-1}\rmd s\big\|^2\Big]+C\bbE\Big[\big\|\int_{r}^{t_k}P^Nu^{S,N}_{s,k-1}\rmd W(s)\big\|^2\Big]\leq C\mathbb{E}[\tau_{k-1}].
 \end{align*}
Thus, 
 $\bbE\big[\|u^{S,N}_{r,k-1}-u^{S,N}_{t_k,k-1}\|^2_{\bbH^2}\big]\leq C\lambda^2_N\mathbb{E}[\tau_{k-1}]\leq C$ for $r\in T_{k-1}$, which together with \eqref{fgeDeltau} and \eqref{u^Ntou^SN} gives 
 \eqref{Eu^NtoEu^SN}.
 
  Hence, we derive 
 \begin{align*}
 \bbE[f(u^{S,N}_{t,\underline{t}})]&\leq C\bbE[f(u^N_0)]+C\int_0^t\bbE[f(u^{S,N}_{r,\underline{r}})]\rmd r+\int_0^T\bbE\big[\|u^N_{\underline{r}}\|^{10}_{\bbH^1}+\|u^{D,N}_{r,\underline{r}}\|^6_{\bbH^1}+\|u^{S,N}_{r,\underline{r}}\|^6_{\bbH^1}+1\big]\rmd r,
 \end{align*}
 which implies 
 \begin{align*}
 \sup_{t\in[0,T]}\bbE[f(u^{S,N}_{t,\underline{t}})]\leq C(\bbE[f(u^N_0)]+1)e^{CT}
 \end{align*}
due to the Gr\"onwall inequality.
 
 Moreover, 
 by utilizing the  supremum type inequalities  as in  \eqref{BDGsup}, one can 
  finish the proof for the case of $p=2.$ For the case of $p>2$, the proof is similar by the use of the It\^o formula and is  omitted. 
\end{proof}
\begin{remark}\label{rmk_exp_regu}
The conclusions in Propositions \ref{expHu^N_m} and \ref{H2regularity} still hold for the solution $\{u^D_{\underline t}(t),\,u^S_{\underline t}(t)\}$ of the split equation  \eqref{u^D_tau_m(t)}  and the solution $\{u^{D,N}_{\underline t}(t),\,u^{S,N}_{\underline t}(t)\}$ of the semi-discrete scheme \eqref{semidiscrete} for $t\in[0,T]$, i.e.,
\begin{align*}&\sup_{t\in[0,T]}\bbE\Big[\exp\Big\{\frac{\HH(u^{S}_{\underline t}(t))}{e^{\alpha_{\lambda}t}}\Big\}+\exp\Big\{\frac{\HH(u^{S,N}_{\underline t}(t))}{e^{\alpha_{\lambda}t}}\Big\}\Big]\leq C,\\
&\bbE\Big[\sup_{t\in[0,T]}\Big(\|u^{D}_{\underline t}(t)\|^p_{\bbH^2}+\|u^{S}_{\underline t}(t)\|^p_{\bbH^2}+\|u^{D,N}_{\underline t}(t)\|^p_{\bbH^2}+\|u^{S,N}_{\underline t}(t)\|^p_{\bbH^2}\Big)\Big]\leq C.
\end{align*}
  The proofs are similar as before by considering $\HH(u^D_{\underline t}(t)),\,\HH(u^{S}_{\underline t}(t)),\,f(u^D_{\underline t}(t)),\,f(u^{S}_{\underline t}(t))$ and those of $u^{D,N}_{\underline t}(t),\,u^{S,N}_{\underline t}(t)$ instead, and hence are omitted. 
\end{remark}

\section{Optimal strong convergence order}\label{section3}
In this section, based on the a.s.-uniform boundedness of the mass, the $\bbH^j\,(j=1,2)$-regularity estimates and the $\bbH^1$-exponential integrability of the numerical solution given in Section \ref{section2}, we show the optimal strong convergence order of the adaptive time-stepping scheme \eqref{Itoform}.
\begin{thm}
Under Assumptions \ref{assump_step1} and \ref{assump_step2}, for $p\ge 2$, there exists a constant $C>0$ such that
\begin{align*}
\sup_{0\leq m\leq M_{T}}\|u(t_{m})-u^N_{m}\|_{L^{p}(\Omega;\bbH)}\leq C(\delta^{\frac{1}{2}}+N^{-2}).
\end{align*}
\end{thm}
\begin{proof}
Noting that $u(t_{m})-u^N_{m}=\big(u^{S,N}_{m-1}(t_{m})-u^N_{m}\big)+\big(u^{S}_{m-1}(t_{m})-u^{S,N}_{m-1}(t_{m})\big)+\big(u(t_{m})-u^{S}_{m-1}(t_{m})\big),$ we split the estimate of the strong error into three steps.  

\textit{Step $1$.}
We first estimate the strong error between  the semi-discrete scheme  and the fully discrete scheme, i.e. $\big\|u^{S,N}_{m-1}(t_{m})-u^N_{m}\big\|_{L^p(\Omega;\bbH)}=:\|{E}_{m}\|_{L^p(\Omega;\bbH)}$. Similarly to the proof of \eqref{Ito_app_1}, the differential form 
\begin{align*}
\rmd (u^{S,N}_{m}(t)-u^{S,N}_{t,m})=-\frac{\epsilon}{2}P^NF_Q(u^{S,N}_{m}(t)-u^{S,N}_{t,m})\rmd t-\bfi \sqrt \epsilon P^N(u^{S,N}_m(t)-u^{S,N}_{t,m})\rmd W(t),
\end{align*}
combining the It\^o formula yields that $\|u^{S,N}_m(t)-u^{S,N}_{t,m}\|^2\leq \|u^{D,N}_m(t_{m+1})-u^{D,N}_{t_{m+1},m}\|^2.$
Since 
\begin{align}\label{u_tau_m-u_tm}
u^{D,N}_{m}(t)-u^{D,N}_{t,m}&=E_{m}+\int_{t_m}^t\bfi \Delta (u^{D,N}_{m}(r)-u^{D,N}_{r,m})\rmd r\notag\\
&\quad +\int_{t_m}^t\bfi\lambda P^N\Big(|u^{D,N}_{m}(r)|^2u^{D,N}_{m}(r)-S^N(r-t_m)|u^{D,N}_{t_m,m}|^2u^{D,N}_{t_m,m}\Big)\rmd r,
\end{align}
we have
\begin{align*}
&\|{E}_{m+1}\|^2\leq \|u^{D,N}_{m}(t_{m+1})-u^{D,N}_{t_{m+1},m}\|^2=\|{E}_m\|^2+2\Big\<{E}_m,\int_{t_m}^{t_{m+1}}\bfi \Delta (u^{D,N}_{m}(t)-u^{D,N}_{t,m})\rmd t\Big\>\\
&+2\Big\<{E}_m,\bfi\lambda \int_{t_m}^{t_{m+1}}P^N(|u^{D,N}_{m}(t)|^2u^{D,N}_{m}(t)-S^N(t-t_m)|u^{D,N}_{t_m,m}|^2u^{D,N}_{t_m,m})\rmd t\Big\>\\
&+\Big\|\int_{t_m}^{t_{m+1}}\bfi \Delta (u^{D,N}_{m}(t)-u^{D,N}_{t,m})\rmd t+\bfi\lambda \int_{t_m}^{t_{m+1}}P^N\Big(|u^{D,N}_{m}(t)|^2u^{D,N}_{m}(t)-S^N(t-t_m)|u^{D,N}_{t_m,m}|^2u^{D,N}_{t_m,m}\Big)\rmd t\Big\|^2\\
&=:\|{E}_m\|^2+II_1+II_2+II_3.
\end{align*}
For the term $II_1,$ using \eqref{u_tau_m-u_tm} and the integration by parts formula, and combining the Gagliardo--Nirenberg inequality $\|u\|^3_{L^6(\mathcal{O})}\leq C\|\nabla u\|\|u\|^2,\,u\in \bbH$ give that 
\begin{align*}
II_1&=2\Big\<\Delta {E}_m,-\int_{t_m}^{t_{m+1}}\int_{t_m}^t\Delta (u^{D,N}_{m}(r)-u^{D,N}_{r,m})\rmd r\rmd t\\
&\quad -\lambda \int_{t_m}^{t_{m+1}}\int_{t_m}^tP^N\Big(|u^{D,N}_{m}(r)|^2u^{D,N}_{m}(r)-S^N(r-t_m)|u^{D,N}_{r,m}|^2u^{D,N}_{r,m}\Big)\rmd r\rmd t\Big\>\\
&\leq C\|{E}_m\|_{\bbH^2}\tau^2_m\Big[\sup_{t\in T_m}\|u^{D,N}_{m}(t)-u^{D,N}_{t,m}\|_{\bbH^2}+\sup_{t\in T_m}\|u^{D,N}_{m}(t)\|^3_{L^6(\OO)}+\sup_{t\in T_m}\|u^{D,N}_{t,m}\|^3_{L^6(\OO)}\Big]\\
&\leq C\tau^2_m\Big[\sup_{t\in T_m}\|u^{D,N}_{m}(t)\|^2_{\bbH^2}+\sup_{t\in T_m}\|u^{D,N}_{t,m}\|^2_{\bbH^2}+1\Big].
\end{align*}
For the term $II_2,$ it follows from the property $\|(S(t)-\mathrm{Id})u\|\leq Ct^{\frac{1}{2}}\|u\|_{\dot{\bbH}^1}$ that
\begin{align*}
II_2&=2\Big\<{E}_m,\bfi\lambda \int_{t_m}^{t_{m+1}}P^N\Big((|u^{D,N}_{m}(t)|^2+|u^{D,N}_{t_m,m}|^2)(u^{D,N}_{m}(t)-u^{D,N}_{t_m,m})+u^{D,N}_{m}(t)u^{D,N}_{t_m,m}(\overline{u^{D,N}_{m}(t)-u^{D,N}_{t_m,m}})\\
&\quad +(\mathrm{Id}-S^N(t-t_m))|u^{D,N}_{t_m,m}|^2u^{D,N}_{t_m,m}\Big)\rmd t\Big\>\\
&\leq C\|{E}_m\|\int_{t_m}^{t_{m+1}}\Big[(\|u^{D,N}_{m}(t)\|^2_{L^{\infty}(\OO)}+\|u^{D,N}_{t_m,m}\|^2_{L^{\infty}(\OO)})\|u^{D,N}_{m}(t)-u^{D,N}_{t_m,m}\|+\tau_m^{\frac{1}{2}}\|u^{D,N}_{t_m,m}\|^3_{\bbH^1}\Big]\rmd t.
\end{align*}
By the inverse inequality $\|P^Nu\|_{\dot{\bbH}^s}\leq \lambda^{\frac{s}{2}}_N\|P^Nu\|,\,u\in\bbH$, $\|u^{D,N}_{m}(t)-u^{D,N}_{m}(t_m)\|\leq \|(S^N(t-t_m)-\mathrm{Id})u^{D,N}_m(t_m)\|+\|\int_{t_m}^tS^N(t-s)P^N|u^{D,N}_m(s)|^2u^{D,N}_m(s)\rmd s\|\leq C(t-t_m)^{\frac{1}{2}}\sup_{t\in T_m}\|u^{D,N}_{m}(t)\|_{\bbH^1}$,  the Minkowskii inequality and the Young inequality, we derive 
\begin{align*}
II_2&\leq C\|{E}_m\|\int_{t_m}^{t_{m+1}}\big(\|u^{D,N}_{m}(t_m)\|^2_{L^{\infty}(\OO)}+\|u^{D,N}_{t_m,m}\|^2_{L^{\infty}(\OO)}+\tau_m\lambda_N^2\sup_{t\in T_m}\|u^{D,N}_{m}(t)\|^2\big)\times \\
&\quad\; (\tau^{\frac{1}{2}}_m\sup_{t\in T_m}\|u^{D,N}_{m}(t)\|_{\bbH^1}+\|E_m\|)\rmd t+C\|E_m\|\int_{t_m}^{t_{m+1}}\tau^{\frac{1}{2}}_m\|u^{D,N}_{t_m,m}\|^3_{\bbH^1}\rmd t\\
&\leq C\tau_m\|{E}_m\|^2(\|u^{D,N}_{m}(t_m)\|^2_{L^{\infty}(\OO)}+\|u^{D,N}_{t_m,m}\|^2_{L^{\infty}(\OO)}+1)+C\tau^2_m(\sup_{t\in T_m}\|u^{D,N}_{m}(t)\|^6_{\bbH^1}+\|u^{D,N}_{t_m,m}\|^6_{\bbH^1}+1).
\end{align*}
For the term $II_3,$ it can be estimated as
\begin{align*}
II_3\leq C\tau^2_m(\sup_{t\in T_m}\|u^{D,N}_{m}(t)\|^2_{\bbH^2}+\sup_{t\in T_m}\|u^{D,N}_{t,m}\|^2_{\bbH^2}).
\end{align*}
Hence, \begin{align*}
\|{E}_{m+1}\|^2&\leq \|{E}_m\|^2+C\tau_m(1+\|u^{D,N}_{m}(t_m)\|^2_{L^{\infty}(\OO)}+\|u^{N}_{m}\|^2_{L^{\infty}(\OO)})\|E_m\|^2\\
&\quad \;+C\tau^2_m(\sup_{t\in T_m}\|u^{D,N}_{m}(t)\|^6_{\bbH^2}+\sup_{t\in T_m}\|u^{D,N}_{t,m}\|^6_{\bbH^2}+1).
\end{align*}
Applying the Gr\"onwall inequality leads to
\begin{align}\label{error1}
\|{E}_{m+1}\|^2&\leq C\sum_{j=0}^m\tau^2_j(\sup_{t\in T_j}\|u^{D,N}_{j}(t)\|^6_{\bbH^2}+\sup_{t\in T_j}\|u^{D,N}_{t,j}\|^6_{\bbH^2}+1)\times\notag\\
&\quad\; \exp\Big\{C\sum_{j=0}^m\tau_j\Big(1+\|u^{D,N}_{j}(t_j)\|^2_{L^{\infty}(\OO)}+\|u^{N}_{j}\|^2_{L^{\infty}(\OO)}\Big)\Big\}.
\end{align}
Note that
\begin{align*}
\Big\|\exp\Big\{C\sum_{j=0}^{M_{T}-1}\tau_j\|\nabla u^{N}_j\|\Big\}\Big\|_{L^{4p}(\Omega)}&\leq \Big\|\exp\Big\{\sum_{j=0}^{M_{T}-1}\tau_j\Big(\rho\|\nabla u^{N}_j\|^2+C(\rho)\Big)\Big\}\Big\|_{L^{4p}(\Omega)}\\
&\leq \Big\|\frac{1}{T}\sum_{j=0}^{M_{T}-1}\tau_j\exp\Big\{T\Big(\rho\|\nabla u^{N}_j\|^2+C(\rho)\Big)\Big\}\Big\|_{L^{4p}(\Omega)}\\
&\leq \frac{1}{T}\sum_{j=0}^{M_{T}-1}T\delta\Big(\bbE\Big[\exp\Big\{4pT\Big(\rho\|\nabla u^{N}_j\|^2+C(\rho)\Big)\Big\}\Big]\Big)^{\frac1{4p}},
\end{align*}
where in the second inequality we use the convexity of $e^{a},a\in\mathbb{R}$, and in the last inequality we use the assumption $\tau_j\leq T\delta$. 
Taking  $\rho=\frac{1}{16pT\exp\{\alpha_{\lambda}T\}},$ and combining $\|\nabla u\|^2\leq 4\mathcal{H}(u)+\|u\|^6,\,u\in \bbH^1$ and Proposition \ref{expHu^N_m} give
\begin{align*}
\bbE\Big[\exp\Big\{4pT\Big(\rho\|\nabla u^{N}_j\|^2+C(\rho)\Big)\Big\}\Big]\leq \bbE\Big[\exp\Big\{\frac{\HH(u^{N}_{j})}{e^{\alpha_{\lambda}t_{j}}}+C(\rho)\Big\}\Big]\leq C,
\end{align*}
which together with the Gagliardo--Nirenberg inequality $\|u\|^2_{L^{\infty}(\OO)}\leq C\|\nabla u\|\|u\|, \,u\in \bbH^1$ implies 
\begin{align}\label{sumexp}
\Big\|\exp\Big\{C\sum_{j=0}^{m}\tau_j\| u^{N}_j\|^2_{L^{\infty}(\OO)}\Big\}\Big\|_{L^{4p}(\Omega)}\leq C.
\end{align}
Similarly, combining Remark \ref{rmk_exp_regu}, one can show that 
\begin{align}\label{sumexp_2}
\Big\|\exp\Big\{C\sum_{j=0}^{m}\tau_j\| u^{D,N}_{j}(t_j)\|^2_{L^{\infty}(\OO)}\Big\}\Big\|_{L^{4p}(\Omega)}\leq C.
\end{align}

Hence, taking the $p$-th power and expectation on both sides of \eqref{error1} and using the assumption $\tau_m\leq T\delta$ lead to
\begin{align*}
\bbE[\|{E}_{m+1}\|^{2p}]&\leq C(T\delta )^p\Big(\bbE\Big|\sum_{j=0}^m\tau_j(\sup_{t\in T_j}\|u^{D,N}_{j}(t)\|^6_{\bbH^2}+\sup_{t\in T_j}\|u^{D,N}_{t,j}\|^6_{\bbH^2}+1)\Big|^{2p}\Big)^{\frac{1}{2}}\times\\
&\quad \;\Big\|\exp\Big\{C\sum_{j=0}^m\tau_j\|u^{D,N}_{j}(t_j)\|^2_{L^{\infty}(\OO)}\Big\}\Big\|_{L^{4p}(\Omega)}^{\frac{p}{2}}\Big\|\exp\Big\{C\sum_{j=0}^m\tau_j\|u^{N}_{j}\|^2_{L^{\infty}(\OO)}\Big\}\Big\|_{L^{4p}(\Omega)}^{\frac{p}{2}} \\
&\leq C\delta^p.
\end{align*}

\textit{Step $2$.}
We estimate the strong error between  the split equation \eqref{u^D_tau_m(t)} and the semi-discrete scheme, i.e., $\big\|u^S_{m-1}(t_m)-u^{S,N}_{m-1}(t_m)\big\|_{L^p(\Omega;\bbH)}=:\|\hat E_m\|_{L^p(\Omega;\bbH)}$. Applying  the chain rule  yields
\begin{align*}
&\quad \|u^{D}_{m}(t_{m+1})-u^{D,N}_{m}(t_{m+1})\|^2=\|u^D_{m}(t_m)-u^{D,N}_{m}(t_m)\|^2\\
&+2\int_{t_m}^{t_{m+1}}\Big\<u^D_{m}(s)-u^{D,N}_{m}(s), \bfi\lambda \big(|u^D_{m}(s)|^2u^D_{m}(s)-P^N|u^{D,N}_{m}(s)|^2u^{D,N}_{m}(s)\big)\Big\>\rmd s,
\end{align*}
and applying the It\^o formula gives
\begin{align*}
&\quad \|u^S_{m}(t_{m+1})-u^{S,N}_{m}(t_{m+1})\|^2=\|u^S_{m}(t_{m})-u^{S,N}_{m}(t_{m})\|^2\\
&+2\epsilon\int_{t_m}^{t_{m+1}}\Big\<u^S_{m}(s)-u^{S,N}_{m}(s),-\frac{1}{2}F_Qu^{S}_{m}(s)+\frac{1}{2}P^NF_Qu^{S,N}_{m}(s)\Big\>\rmd s\\
&+2\sqrt\epsilon \int_{t_m}^{t_{m+1}}\Big\<u^S_{m}(s)-u^{S,N}_{m}(s),-\bfi (u^S_{m}(s)-P^Nu^{S,N}_{m}(s))\rmd W(s)\Big\>\\
&+\epsilon\int_{t_m}^{t_{m+1}}\sum_{k=1}^{\infty}\|(u^S_{m}(s)-P^Nu^{S,N}_{m}(s))Q^{\frac{1}{2}}e_k\|^2\rmd s.
\end{align*}
Therefore, we derive
\begin{align*}
\|\hat E_{m+1}\|^2&=\|\hat E_m\|^2+2\int_{t_m}^{t_{m+1}}\Big\<u^D_{m}(s)-u^{D,N}_{m}(s), \bfi\lambda \big(|u^D_{m}(s)|^2u^D_{m}(s)-P^N|u^{D,N}_{m}(s)|^2u^{D,N}_{m}(s)\big)\Big\>\rmd s\\
&\quad +2\epsilon\int_{t_m}^{t_{m+1}}\Big\<u^S_{m}(s)-u^{S,N}_{m}(s),-\frac{1}{2}F_Qu^{S}_{m}(s)+\frac{1}{2}P^NF_Qu^{S,N}_{m}(s)\Big\>\rmd s\\
&\quad +2\sqrt \epsilon \int_{t_m}^{t_{m+1}}\Big\<u^S_{m}(s)-u^{S,N}_{m}(s),-\bfi (u^S_{m}(s)- P^Nu^{S,N}_{m}(s))\rmd W(s)\Big\>\\
&\quad +\epsilon\int_{t_m}^{t_{m+1}}\sum_{k=1}^{\infty}\|(u^S_{m}(s)-P^Nu^{S,N}_{m}(s))Q^{\frac{1}{2}}e_k\|^2\rmd s=:\|\hat E_m\|^2+III_1+III_2+III_3+III_4.
\end{align*}
For the term $III_1,$ combining the cubic difference formula  gives
\begin{align*}
|III_1|&\leq C\int_{t_m}^{t_{m+1}}\|u^D_{m}(s)-u^{D,N}_{m}(s)\|\Big[\big(\|u^D_{m}(s)\|^2_{L^{\infty}(\OO)}+\|u^{D,N}_{m}(s)\|^2_{L^{\infty}(\OO)}\big)\|u^D_{m}(s)-u^{D,N}_{m}(s)\|\\
&\quad+\|(\mathrm{Id}-P^N)|u^{D,N}_{m}(s)|^2u^{D,N}_{m}(s)\|\Big]\rmd s\\
&\leq C\int_{t_m}^{t_{m+1}}\|u^D_{m}(s)-u^{D,N}_{m}(s)\|^2\big(\|u^D_{m}(s)\|^2_{L^{\infty}(\OO)}+\|u^{D,N}_{m}(s)\|^2_{L^{\infty}(\OO)}+1\big)\rmd s\\
&\quad +C\lambda_N^{-2}\int_{t_m}^{t_{m+1}}\|u^{D,N}_{m}(s)\|^6_{\bbH^2}\rmd s.
\end{align*}
By the properties $\|u^D_{m}(s)-u^D_{m}(t_m)\|^2\leq C\tau_m\|u^D_{m}(t_m)\|^2_{\bbH^1}$ and $\|u^D_{m}(s)-u^{D}_{m}(t_m)\|^2_{\bbH^1}\leq C\tau_m\|u^D_{m}(t_m)\|^2_{\bbH^2}$ of $u^D_{m}(s)$, and those of $u^{D,N}_{m}(s)$ for $s\in T_m,$ we arrive at
\begin{align*}
|III_1|&\leq C\int_{t_m}^{t_{m+1}}(\|\hat E_m\|^2+\tau_m\|u^D_{m}(t_m)\|^2_{\bbH^1}+\tau_m\|u^{D,N}_{m}(t_m)\|^2_{\bbH^1})(\|u^D_{m}(t_m)\|^2_{L^{\infty}(\OO)}+\|u^{D,N}_{m}(t_m)\|^2_{L^{\infty}(\OO)}\\
&\quad +\tau_m\|u^D_{m}(t_m)\|^2_{\bbH^2}+\tau_m\|u^{D,N}_{m}(t_m)\|^2_{\bbH^2}+1)+C\lambda_N^{-2}\int_{t_m}^{t_{m+1}}\|u^{D,N}_{m}(s)\|^6_{\bbH^2}\rmd s\\
&\leq C\tau_m\|\hat E_m\|^2(\|u^D_{m}(t_m)\|^2_{\bbH^1}+\|u^{D,N}_{m}(t_m)\|^2_{\bbH^1}+1)\\
&\quad +C\int_{t_m}^{t_{m+1}}\tau_m(\|u^{D}_{m}(t_m)\|^4_{\bbH^1}+\|u^{D,N}_{m}(t_m)\|^4_{\bbH^1})+\lambda^{-2}_N\|u^{D,N}_{m}(s)\|^6_{\bbH^2}\rmd s.
\end{align*}
Terms $III_2$ and $III_4$ can be estimated respectively as
\begin{align*}
|III_2|&\leq C\int_{t_m}^{t_{m+1}}\Big(\|u^{S}_{m}(s)-u^{S,N}_{m}(s)\|^2(\|F_Q\|_{L^{\infty}(\OO)}+1)+\lambda^{-2}_N\|F_Q\|^2_{H^2}\|u^{S,N}_{m}(s)\|^2_{\bbH^2}\Big)\rmd s
\end{align*}
and 
\begin{align*}
|III_4|\leq C\int_{t_m}^{t_{m+1}}\Big(\|u^S_{m}(s)-u^{S,N}_{m}(s)\|^2\|Q^{\frac{1}{2}}\|^2_{\mathcal{L}^1_2}+\lambda^{-2}_N\|u^{S,N}_{m}(s)\|^2_{\bbH^2}\|Q^{\frac{1}{2}}\|^2_{\mathcal{L}^2_2}\Big)\rmd s.
\end{align*}
By the H\"older continuity and the triangle inequality, we obtain
\begin{align*}
|III_2|+|III_4|&\leq  C\int_{t_m}^{t_{m+1}}\Big[\|\hat E_m\|^2+
\|u^S_m(s)-u^S_m(t_m)\|^2+\|u^{S,N}_{m}(s)-u^{S,N}_m(t_m)\|^2\Big]\rmd s\\
&\quad +C\int_{t_m}^{t_{m+1}}\lambda^{-2}_N\|u^{S,N}_{m}(s)\|^2_{\bbH^2}\rmd s.
\end{align*}
Combining estimates of terms $III_j,\,j=1,2,4$ yields that
\begin{align*}
\|\hat E_{m+1}\|^2&\leq \|\hat E_m\|^2+C\tau_m\|\hat E_m\|^2(\|u^D_{m}(t_m)\|^2_{\bbH^1}+\|u^{D,N}_{m}(t_m)\|^2_{\bbH^1}+1)\\
&\quad +C\int_{t_m}^{t_{m+1}}\Big[\tau_m(\|u^{D}_{m}(t_m)\|^4_{\bbH^1}+\|u^{D,N}_{m}(t_m)\|^4_{\bbH^1})+
\|u^S_m(s)-u^S_m(t_m)\|^2\\
&\quad +\|u^{S,N}_{m}(s)-u^{S,N}_m(t_m)\|^2+\lambda^{-2}_N(\|u^{D,N}_{m}(s)\|^6_{\bbH^2}+\|u^{S,N}_{m}(s)\|^2_{\bbH^2})\Big]\rmd s+III_3.
\end{align*}
By iteration, we have
\begin{align*}
\|\hat E_{m+1}\|^{2}&\leq \|\hat E_0\|^{2}+C\sum_{j=0}^m\tau_j\|\hat E_j\|^{2}(\|u^D_{j}(t_j)\|^2_{\bbH^1}+\|u^{D,N}_{j}(t_j)\|^2_{\bbH^1}+1) +C\int_0^{t_{m+1}}\Big[(\delta+\lambda^{-2}_N)\times \\
&\quad(\|u^{D,N}_{\underline{s}}(t_{\underline s})\|^6_{\bbH^2}+\|u^{D}_{\underline{s}}(t_{\underline s})\|^4_{\bbH^1}+\|u^{S,N}_{\underline{s}}(s)\|^2_{\bbH^2})+
\|u^S_{\underline s}(s)-u^S_{\underline s}(t_{\underline s})\|^2 +\|u^{S,N}_{\underline s}(s)-u^{S,N}_{\underline s}(t_{\underline s})\|^2\Big]\rmd s\\
&\quad +2\int_0^{t_{m+1}}\Big\<(\mathrm{Id}-P^N)(u^{S}_{\underline{s}}(s)-u^{S,N}_{\underline{s}}(s)),-\bfi (\mathrm{Id}-P^N)u^{S,N}_{\underline{s}}(s)\rmd W(s)\Big\>\\
&=:\|\hat E_0\|^2+C\sum_{j=0}^m\tau_j\|\hat E_j\|^{2}(\|u^D_{j}(t_j)\|^2_{\bbH^1}+\|u^{D,N}_{j}(t_j)\|^2_{\bbH^1}+1)+J_1+J_2,
\end{align*}
which implies
\begin{align*}
\|\hat E_{m+1}\|^2\leq (\lambda^{-2}_N\|u^N_0\|^2_{\bbH^2}+J_1+J_2)\exp\Big\{C\sum_{j=0}^m\tau_j(\|u^D_{j}(t_j)\|^2_{\bbH^1}+\|u^{D,N}_{j}(t_j)\|^2_{\bbH^1}+1)\Big\}.
\end{align*}
Taking $L^p(\Omega)$-norm, and noticing that the H\"older continuity of $u^S_{\underline {\cdot}}, u^{S,N}_{\underline {\cdot}}$ and the Burkholder--Davis--Gundy inequality give 
\begin{align*}
\|J_1+J_2\|_{L^{2p}(\Omega)}\leq&\, C(\delta+\lambda_N^{-2})+C\int_0^T\big(\|u^S_{\underline s}(s)-u^S_{\underline s}(t_{\underline s})\|_{L^{4p}(\Omega;\bbH)}^2+\|u^{S,N}_{\underline s}(s)-u^{S,N}_{\underline s}(t_{\underline s})\|^2_{L^{4p}(\Omega;\bbH)}\big)\rmd s\\
&+C\bbE\Big[\Big(\int_0^T\lambda^{-4}_N(\|u^S_{\underline{s}}(s)\|^2_{\bbH^2}+\|u^{S,N}_{\underline{s}}(s)\|^2_{\bbH^2})\|u^{S,N}_{\underline{s}}(s)\|^2_{\bbH^2}\|Q^{\frac{1}{2}}\|^2_{\mathcal{L}^1_2}\rmd s\Big)^{p}\Big]^{\frac{1}{2p}}\\
\leq &\,C(\delta+\lambda_N^{-2}).
\end{align*}
 Moreover, one can show that 
$$\Big\|\exp\Big\{C\sum_{j=0}^m\tau_j(\|u^D_{j}(t_j)\|^2_{\bbH^1}+\|u^{D,N}_{j}(t_j)\|^2_{\bbH^1}+1)\Big\|_{L^{2p}(\Omega)}\leq C,$$ whose proof is similar to that of \eqref{sumexp}--\eqref{sumexp_2} and is omitted. 
Hence, we arrive at $\bbE[\|\hat E_{m+1}\|^{2p}]\leq C(\lambda^{-2}_N+\delta)^p.$ 

\textit{Step $3$.}
For the strong error between the original stochastic Schr\"odinger equation \eqref{Itosense} and the split equation \eqref{u^D_tau_m(t)}, i.e., $\big\|u(t_m)-u^S_{m-1}(t_m)\big\|_{L^p(\Omega;\bbH)},$ it follows from \cite[Theorem $2.2$]{Cui_2019}  that 
$\big\|u(t_m)-u^S_{m-1}(t_m)\big\|_{L^p(\Omega;\bbH)}\leq C\delta^{\frac12}.$

Combining \textit{Steps} $1$-$3$ finishes the proof.
\end{proof}

\begin{remark}
In practice, instead of verifying whether a timestep function satisfies the low bound in Assumption \ref{assump_step1} \eqref{condition2}, people usually introduce a backstop scheme with a uniform timestep and couple it with \eqref{Itoform} to ensure that a simulation over the interval [0,T] can be completed in a finite number of timesteps; see 
e.g. \cite{adapAC} and references therein for more details. 
\end{remark}

\section{Numerical asymptotics}
In this section, we study the asymptotic behavior of the adaptive time-stepping fully discrete scheme  \eqref{Itoform} for 
 the  stochastic NLS equation \eqref{schrod}  as the noise intensity  $\epsilon$ tends to zero. Note that the dependence on $\epsilon$ of solutions is emphasized in this section, for example,  solutions of \eqref{schrod} and \eqref{Itoform} are denoted by $\{u^{\epsilon}(t):\,t\in[0,T]\}$ and  $\{u^{D,N,\epsilon}_{t,\underline{t}},\,u^{S,N,\epsilon}_{t,\underline{t}}:\,t\in[0,T]\}$, respectively. The tool for this study is the theory of large deviation, which  describes  precisely the weak convergence  towards the Dirac measure on the solution   of the corresponding  skeleton equation as $\epsilon\to0.$  We refer to e.g. \cite{FQ, gautier2005, Gautier_2005} for the study on the LDP of the solution of \eqref{schrod}.

Set $H_0:=Q^{\frac{1}{2}}H.$ 
Then $H_0$ is a Hilbert space with the inner product $\<u,v\>_{H_0}:=\<Q^{-\frac{1}{2}}u,Q^{-\frac{1}{2}}v\>_{H}$ and the induced norm $\|\cdot\|_{H_0}^2=\<\cdot,\cdot\>_{H_0},$  where $Q^{-\frac{1}{2}}$ is the pseudo inverse of $Q^{\frac{1}{2}}.$ Denote $\mathcal{S}_M:=\big\{\nu\in L^2([0,T];H_0)\big|\int_0^T\|\nu(s)\|^2_{H_0}\rmd s\leq M\big\}$ and $\mathcal{P}_M:=\big\{\nu:\Omega\times [0,T]\to H_0\big|\nu \text{ is }\mathcal{F}_t\text{-predictable and } \nu\in \mathcal{S}_M \text{ a.s.}\big\}$ for each $M\in(0,\infty).$ It can be checked that $\mathcal{S}_M$ is a compact Polish space endowed with the weak topology $d_1(g_1,g_2)=\sum_{k\ge 1}\frac{1}{2^k}|\int_0^T\<g_1(s)-g_2(s),\xi_k(s)\>_{H_0}\rmd s|,$ where $\{\xi_k\}_{k\ge 1}$ is an orthogonal basis
of $L^2([0,T];H_0);$ see e.g. \cite[Section $4$]{BD2} and \cite[Section $2$]{FQ}. 
In the sequel, we denote by $\xrightarrow[]{d}$ the convergence in distribution.

In order to establish the LDP for the solution of \eqref{Itoform}, we consider 
the following stochastic controlled equation
\begin{equation}
\label{small_noise}
\begin{cases}
\rmd u^{D,N,\epsilon}_{\nu^{\epsilon},m}(t)=\bfi \Delta u^{D,N,\epsilon}_{\nu^{\epsilon},m}(t)\rmd t+\bfi \lambda S^N(t-t_m)P^N|u^{N,\epsilon}_{\nu^{\epsilon},m}|^2u^{N,\epsilon}_{\nu^{\epsilon},m}\rmd t,
\quad u^{D,N,\epsilon}_{\nu^{\epsilon},m}(t_m)=u^{N,\epsilon}_{\nu^{\epsilon},m},\vspace{1mm}\\
\rmd u^{S,N,\epsilon}_{\nu^{\epsilon},m}(t)=-\frac{\epsilon }{2}P^NF_Qu^{S,N,\epsilon}_{\nu^{\epsilon},m}(t)\rmd t-\bfi P^Nu^{S,N,\epsilon}_{\nu^{\epsilon},m}(t)\nu^{\epsilon}(t)\rmd t\vspace{1mm}\\
 \qquad \quad \qquad \;\;
-\bfi \sqrt{\epsilon}P^Nu^{S,N,\epsilon}_{\nu^{\epsilon},m}\rmd W(t), \quad 
u^{S,N,\epsilon}_{\nu^{\epsilon},m}(t_m)=u^{D,N,\epsilon}_{\nu^{\epsilon},m}(t_{m+1}),
\end{cases}
\end{equation}
and the skeleton equation 
\begin{equation}\label{deterministic}
\begin{cases}
\rmd w^{D,N}_{\nu,m}(t)=\bfi \Delta w^{D,N}_{\nu,m}(t)\rmd t+\bfi \lambda S^N(t-t_m)P^N |w^{N}_{\nu,m}|^2w^{N}_{\nu,m}\rmd t,\quad w^{D,N}_{\nu,m}(t_m)=w^{N}_{\nu,m},\vspace{1mm}\\
\rmd w^{S,N}_{\nu,m}(t)=-\bfi P^Nw^{S,N}_{\nu,m}(t)\nu(t)\rmd t,\quad w^{S,N}_{\nu,m}(t_m)=w^{D,N}_{\nu,m}(t_{m+1})
\end{cases}
\end{equation}
for $t\in T_m$ with $\nu^{\epsilon}, \nu\in L^2([0,T];H_0).$
  Here, the initial data both are $u^N_0.$
Define measurable maps $\mathcal{G}^{\epsilon},\mathcal{G}^0:\mathcal{C}([0,T];H)\to\mathcal{C}([0,T];\bbH_N)$  by $\mathcal{G}^{\epsilon}\big(\sqrt{\epsilon}W+\int_0^{\cdot}\nu^{\epsilon}(s)\rmd s\big):=u^{S,N,\epsilon}_{\nu^{\epsilon},\underline{\cdot}}(\cdot)$ and $\mathcal{G}^0\big(\int_0^{\cdot}\nu(s)\rmd s\big):=w^{S,N}_{\nu,\underline{\cdot}}(\cdot).$ And denote $u^{N,\epsilon}_{\nu^{\epsilon},m+1}:=u^{S,N,\epsilon}_{\nu^{\epsilon},m}(t_{m+1})$ and $w^N_{\nu,m+1}:=w^{S,N}_{\nu,m}(t_{m+1}).$

Similar assumptions to Assumptions \ref{assump_step1} and \ref{assump_step2} are given as follows.
\begin{assumption}\label{refined1}
Let $\tau_m$ satisfy 
\begin{align*}
&\tau_m\leq \min\big\{L_1 \|w^N_{\nu,m}\|^2\|w^N_{\nu,m}\|^{-6}_{L^6(\mathcal{O})},\;L_1\|u^{N,\epsilon}_{\nu^{\epsilon},m}\|^2\|u^{N,\epsilon}_{\nu^{\epsilon},m}\|^{-6}_{L^6(\OO)},\;T\delta \big\}\quad a.s.,\\
&\tau_m\ge \max\big\{(\zeta \|u^{N,\epsilon}_{\nu^{\epsilon},m}\|^{\beta}+\xi)^{-1}\delta,\;(\zeta \|w^N_{\nu,m}\|^{\beta}+\xi)^{-1}\delta\big\}\quad a.s.
\end{align*}
with constants $L_1,\zeta,\beta,\xi>0$ and small constant $\delta\in(0,1)$ independent of $\epsilon$.
\end{assumption} 

\begin{assumption}\label{refined2}
Let $\tau_m$ satisfy 
\begin{align*}
&\tau_m^{\frac{1}{2}-\gamma}\lambda_N\leq L_2\quad a.s.,\\
&\tau_m^{\gamma}\max\{\HH(u^{N,\epsilon}_{\nu^{\epsilon},m}),\HH(w^N_{\nu,m})\}\leq L_3\quad a.s.
\end{align*}
for some $\gamma\in(0,\frac{1}{2})$ and constants $L_2,L_3>0$ independent of $\epsilon.$
\end{assumption}

The main result of this section is stated as follows.
\begin{thm}\label{mainLDP}
Under Assumptions \ref{refined1} and \ref{refined2}, the family $\{u^{S,N,\epsilon}_{\cdot,\underline{\cdot}}\}_{\epsilon\in(0,1)}$  of solutions of \eqref{Itoform} satisfies the LDP on $\mathcal{C}([0,T];\bbH_N)$, i.e., 
\begin{itemize}
\item[(\romannumeral1)] for each closed subset $F$ of $\mathcal{C}([0,T];\bbH_N),$ 
\begin{align*}
\limsup_{\epsilon\to0}\epsilon\log\mathbb{P}(u^{S,N,\epsilon}_{\cdot,\underline{\cdot}}\in F)\leq -\inf_{x\in F}I(x);
\end{align*}
\item[(\romannumeral2)] for each open subset $G$ of $\mathcal{C}([0,T];\bbH_N),$ 
\begin{align*}
\liminf_{\epsilon\to0}\epsilon\log \mathbb{P}(u^{S,N,\epsilon}_{\cdot,\underline{\cdot}}\in G)\ge -\inf_{x\in G}I(x),
\end{align*}
\end{itemize}
where  the rate function $I:\mathcal{C}([0,T];\bbH_N)\to [0,\infty]$ is defined by
\begin{align*}
I(f)=\inf_{\{\nu\in L^2([0,T];H_0):f=\mathcal{G}^0(\int_0^{\cdot}\nu (s)\rmd s)\}}\frac{1}{2}\|\nu(s)\|^2_{H_0}\rmd s.
\end{align*}
\end{thm}

Below we give the a.s.-uniform boundedness of the masses,
and the $\bbH^1$-regularity estimates of solutions of \eqref{small_noise} and \eqref{deterministic}, which are similar  to those of the fully discrete scheme \eqref{Itoform}, i.e., Lemma \ref{Hnorm} and Proposition \ref{H^1regularity}.

\begin{prop}
Let $M>0,$ and let $\{\nu^{\epsilon}\}_{\epsilon\in(0,1)}\subset \mathcal{P}_M$. Under Assumptions \ref{refined1} and \ref{refined2}, 
\begin{align}
\sup_{\epsilon\in(0,1)}\sup_{t\in[0,T]}\big(\|u^{D,N,\epsilon}_{\nu^{\epsilon},\underline{t}}\|^2\vee \|u^{S,N,\epsilon}_{\nu^{\epsilon},\underline{t}}\|^2\big)\leq e^{L_1T}\|u^{N}_{0}\|^2\quad  a.s.,\label{ldp_regu_1}
\end{align}
and for $p\ge 2,$ 
\begin{align}
\sup_{\epsilon\in(0,1)}\mathbb{E}\Big[\sup_{t\in[0,T]}\|u^{S,N,\epsilon}_{\nu^{\epsilon},\underline{t}}(t)\|^p_{\bbH^1}\Big]\leq C\quad a.s.,\label{ldp_regu_2}
\end{align}
where the constant  $L_1$ is given in Assumption \ref{refined1} and $C:=C(p,T,M,\HH(u^N_0))>0$.
\end{prop}
\begin{proof}
 For the proof of \eqref{ldp_regu_1}, we note that 
$\<u^{S,N,\epsilon}_{\nu^{\epsilon},m}(s),-\bfi P^Nu^{S,N,\epsilon}_{\nu^{\epsilon},m}(s) \nu^{\epsilon}(s)\>=0.$ A similar proof to that of Lemma \ref{Hnorm} leads to \eqref{ldp_regu_1}.

For the proof of \eqref{ldp_regu_2}, similar to the proof of \eqref{iterH_D}, we have
$$\mathcal{H}(u^{D,N,\epsilon}_{\nu^{\epsilon},m}(t_{m+1}))\leq \mathcal{H}(u^{N,\epsilon}_{\nu^{\epsilon},m})+C\tau^{1+\gamma}_m(\mathcal{H}(u^{N,\epsilon}_{\nu^{\epsilon},m})+1)\leq \mathcal{H}(u^{N,\epsilon}_{\nu^{\epsilon},m})+C\tau_m$$
under Assumptions \ref{refined2}.
Applying the It\^o formula to $\HH(u^{S,N,\epsilon}_{\nu^{\epsilon},\underline{\cdot}}(\cdot))$, and noticing that $$\big\<\nabla u^{S,N,\epsilon}_{\nu^{\epsilon},m}(s),\bfi \nabla \big(u^{S,N,\epsilon}_{\nu^{\epsilon},m}(s)\nu^{\epsilon}(s)\big)\big\>=\big\<\nabla u^{S,N,\epsilon}_{\nu^{\epsilon},m}(s),\bfi u^{S,N,\epsilon}_{\nu^{\epsilon},m}(s)\nabla \nu^{\epsilon}(s)\big\>,$$ we derive
\begin{align}\label{nuregu}
&\quad \mathbb{E}[\mathcal{H}(u^{S,N,\epsilon}_{\nu^{\epsilon},m}(t))]-\mathbb{E}[\mathcal{H}(u^{D,N,\epsilon}_{\nu^{\epsilon},m}(t_{m+1}))]\notag\\
&\leq C\mathbb{E}\int_{t_m}^t\|\nabla u^{S,N,\epsilon}_{\nu^{\epsilon},m}(s)\|\Big[\|\nabla u^{S,N,\epsilon}_{\nu^{\epsilon},m}(s)\|\|F_Q\|_{L^{\infty}(\mathcal{O})}+\|u^{S,N,\epsilon}_{\nu^{\epsilon},m}(s)\|_{L^{\infty}(\mathcal{O})}\|\nabla F_Q\|+\|u^{S,N,\epsilon}_{\nu^{\epsilon},m}(s)\|\|\nabla \nu^{\epsilon}(s)\|_{L^{\infty}(\mathcal{O})}\Big]\rmd s\notag\\
&\quad +C\mathbb{E}\int_{t_m}^t\|u^{S,N,\epsilon}_{\nu^{\epsilon},m}(s)\|^3_{L^6(\mathcal{O})}\|u^{S,N,\epsilon}_{\nu^{\epsilon},m}(s)\|(\|F_Q\|_{L^{\infty}(\mathcal{O})}+\|\nu^{\epsilon}(s)\|_{L^{\infty}(\mathcal{O})})\rmd s +C\mathbb{E}\int_{t_m}^t\|u^{S,N,\epsilon}_{\nu^{\epsilon},m}(s)\|^2_{\bbH^1}\rmd s\notag\\
&\leq C\mathbb{E}\int_{t_m}^t(\| \nabla u^{S,N,\epsilon}_{\nu^{\epsilon},m}(s)\|^{2}+1)\rmd s+C\mathbb{E}\int_{t_m}^t\frac{1}{M}(\|\nabla \nu^{\epsilon}(s)\|_{L^{\infty}(\mathcal{O})}^2+\|\nu^{\epsilon}(s)\|^2_{L^{\infty}(\mathcal{O})})\rmd s,
\end{align}
where in the second inequality we use the Young inequality.
 By iteration and combining $\mathcal{H}(u^{S,N,\epsilon}_{\nu^{\epsilon},\underline t}(t))\ge \frac14(\|\nabla u^{S,N,\epsilon}_{\nu^{\epsilon},\underline t}(t)\|^2-\|u^{S,N,\epsilon}_{\nu^{\epsilon},\underline t}(t)\|^6)$, we obtain
\begin{align*}
\mathbb{E}[\mathcal{H}(u^{S,N,\epsilon}_{\nu^{\epsilon},\underline{t}}(t))]&\leq \mathbb{E}[\mathcal{H}(u^N_0)]+C\mathbb{E}\int_0^t\mathcal{H}(u^{S,N,\epsilon}_{\nu^{\epsilon},\underline{s}}(s))\rmd s+CT\\
&\quad +C\mathbb{E}\int_0^t\frac{1}{M}(\|\nabla \nu^{\epsilon}(s)\|_{L^{\infty}(\mathcal{O})}^2+\|\nu^{\epsilon}(s)\|^2_{L^{\infty}(\mathcal{O})})\rmd s.
\end{align*}
It follows from $\|v\|_{H^2}\leq \|Q^{\frac{1}{2}}\|_{\mathcal{L}(H,H^2)}\|Q^{-\frac{1}{2}}v\|\leq \|Q^{\frac{1}{2}}\|_{\mathcal{L}^2_2}\|v\|_{H_0}$ for $v\in H_0$  that 
$$\int_{0}^T(\|\nabla \nu^{\epsilon}(s)\|^2_{L^{\infty}(\mathcal{O})}+\|\nu^{\epsilon}(s)\|^2_{L^{\infty}(\mathcal{O})})\rmd s\leq C\int_0^T\|\nu^{\epsilon}(s)\|^2_{H^2}\rmd s\leq C\int_0^T\|\nu^{\epsilon}(s)\|^2_{H_0}\rmd s\leq CM\quad a.s.$$  This leads to
\begin{align*}
\mathbb{E}[\mathcal{H}(u^{S,N,\epsilon}_{\nu^{\epsilon},\underline{t}}(t))]
&\leq \mathbb{E}[\mathcal{H}(u^N_0)]+C\Big(\mathbb{E}\int_0^t\mathcal{H}(u^{S,N,\epsilon}_{\nu^{\epsilon},\underline{s}}(s))\rmd s+1\Big).
\end{align*} 
Hence, 
\begin{align*}
\sup_{t\in[0,T]}\mathbb{E}[\mathcal{H}(u^{S,N,\epsilon}_{\nu^{\epsilon},\underline{t}}(t))]\leq \Big(\mathbb{E}[\mathcal{H}(u^N_0)]+C\Big)e^{CT}.
\end{align*}
The remaining proof is similar to that of Proposition \ref{H^1regularity} and hence is omitted. 
\end{proof}

\begin{prop}\label{propwregu}
Let $M>0,$ and let $\nu\in\mathcal{S}_M$. Under Assumptions \ref{refined1} and \ref{refined2}, $$\sup_{t\in[0,T]}\|w^N_{\nu,\underline{t}}\|^2\leq e^{L_1T}\|w^N_{0}\|^2,\quad \sup_{t\in[0,T]}\|w^{S,N}_{\nu,\underline{t}}(t)\|^2_{\mathbb{H}^1}\leq C\quad a.s.,$$
where the constant $L_1$ is given in Assumption \ref{refined1} and $C:=C(T,M,\HH(u^N_0))>0$.
\end{prop}
\begin{proof}
 It is clear that
\begin{align*}
\|w^{S,N}_{\nu,m}(t_{m+1})\|^2=\|w^{S,N}_{\nu,m}(t_m)\|^2=\|w^{D,N}_{\nu,m}(t_{m+1})\|^2,
\end{align*} 
which combining $\|w^{D,N}_{\nu,m}(t_{m+1})\|^2\leq (1+L_1\tau_m)\|w^{N}_{\nu,m}\|^2$ implies that
\begin{align*}
\|w^N_{\nu,m}\|^2\leq e^{L_1T}\|w^N_{\nu,0}\|^2.
\end{align*}

 Similar to the proof of \eqref{iterH_D}, we have
\begin{align*}
\mathcal{H}(w^{D,N}_{\nu,m}(t_{m+1}))\leq \mathcal{H}(w^N_{\nu,m})+C\tau^{\gamma+1}_m(\mathcal{H}(w^{N}_{\nu,m})+1)\leq \mathcal{H}(w^N_{\nu,m})+C\tau_m
\end{align*}
under Assumptions \ref{refined2}. Applying the chain rule and the Young inequality gives
 \begin{align*}
\mathcal{H}(w^{S,N}_{\nu,m}(t))-\mathcal{H}(w^{S,N}_{\nu,m}(t_m))
&=\int_{t_m}^{t}\Big\langle \nabla w^{S,N}_{\nu,m}(s),\nabla \big(-\bfi P^Nw^{S,N}_{\nu,m}(s)\nu(s)\big)\Big\rangle\rmd s\\
&\quad -\lambda \int_{t_m}^{t}\Big\langle |w^{S,N}_{\nu,m}(s)|^2w^{S,N}_{\nu,m}(s),-\bfi P^N w^{S,N}_{\nu,m}(s)\nu(s)\Big\rangle\rmd s\\
&\leq C\int_{t_m}^{t}\|\nabla w^{S,N}_{\nu,m}(s)\|\|\nu(s)\|_{H^2}\rmd s.
\end{align*}
Hence, 
\begin{align*}
\mathcal{H}(w^{S,N}_{\nu,m}(t))&\leq \mathcal{H}(w^{D,N}_{\nu,m}(t_{m+1}))+C\int_{t_m}^t(\|\nabla w^{S,N}_{\nu,m}(s)\|^2+\|\nu(s)\|^2_{H^2})\rmd s\\
&\leq \mathcal{H}(w^N_{\nu,m})+C\int_{t_m}^t(\mathcal{H}(w^{S,N}_{\nu,m}(s))+\|\nu(s)\|^2_{H^2}+1)\rmd s,
\end{align*}
which together with  the iteration and the fact that $\int_0^T\|\nu^{\epsilon}(s)\|^2_{H^2}\rmd s\leq C\int_0^T\|\nu^{\epsilon}(s)\|^2_{H_0}\rmd s\leq CM$ yields
\begin{align*}
\mathcal{H}(w^{S,N}_{\nu,m}(t))\leq \mathcal{H}(w^N_{\nu,0})+C\int_0^t\mathcal{H}(w^{S,N}_{\nu,m}(s))\rmd s+CT+CM.
\end{align*}
Applying the Gr\"onwall  inequality finishes the proof.
\end{proof}

\begin{prop}\label{laplace1}
Let $M>0.$ Under Assumptions \ref{refined1} and \ref{refined2}, the set $K_M:=\big\{\mathcal{G}^0(\int_0^{\cdot}\nu(s)\rmd s):\nu\in\mathcal{S}_M\big\}$  is a compact subset in $\mathcal{C}([0,T];\bbH_N)$.
\end{prop}
\begin{proof}
It suffices to prove that $K_M$ is sequentially compact in $\mathcal{C}([0,T];\bbH_N).$ Let $\{\nu^{\epsilon},\nu\}\subset \mathcal{S}_M$ with $\nu^{\epsilon}\to\nu$ in $\mathcal{S}_M$. The property $\|S^N(t)\|_{\mathcal{L}(\bbH;\bbH)}=1$ and Proposition \ref{propwregu} imply
\begin{align*}
&\quad \;\|w^{D,N}_{\nu^{\epsilon},m}(t_{m+1})-w^{D,N}_{\nu,m}(t_{m+1})\|^2\\
&=\|w^N_{\nu^{\epsilon},m}-w^N_{\nu,m}\|^2+2\tau_m\big\<w^N_{\nu^{\epsilon},m}-w^N_{\nu,m},\bfi \lambda \big(|w^N_{\nu^{\epsilon},m}|^2w^N_{\nu^{\epsilon},m}-|w^N_{\nu,m}|^2w^N_{\nu,m}\big)\big\>\\
&\quad\; +\tau_m^2\big\||w^N_{\nu^{\epsilon},m}|^2w^N_{\nu^{\epsilon},m}-|w^N_{\nu,m}|^2w^N_{\nu,m}\big\|^2\\
&\leq \|w^N_{\nu^{\epsilon},m}-w^N_{\nu,m}\|^2+C\tau_m\|w^N_{\nu^{\epsilon},m}-w^N_{\nu,m}\|^2\Big[\|w^N_{\nu^{\epsilon},m}\|^2_{\bbH^1}+\|w^N_{\nu,m}\|^2_{\bbH^1}+\tau_m(\|w^N_{\nu^{\epsilon},m}\|^4_{\bbH^1}+w^N_{\nu,m}\|^4_{\bbH^1})\Big]\\
&\leq \|w^N_{\nu^{\epsilon},m}-w^N_{\nu,m}\|^2+C\tau_m\|w^N_{\nu^{\epsilon},m}-w^N_{\nu,m}\|^2\quad a.s.
\end{align*}
Note that for $t\in T_m,$
\begin{align*}
\frac{\rmd}{\rmd t}(w^{S,N}_{\nu^{\epsilon},m}(t)-w^{S,N}_{\nu,m}(t))=-\bfi P^Nw^{S,N}_{\nu^{\epsilon},m}(t)\nu^{\epsilon}(t)+\bfi P^Nw^{S,N}_{\nu,m}(t)\nu(t).
\end{align*}
 By the chain rule, we have for $t\in T_m,$
\begin{align*}
\|w^{S,N}_{\nu^{\epsilon},m}(t)-w^{S,N}_{\nu,m}(t)\|^2&=\|w^{D,N}_{\nu^{\epsilon},m}(t_{m+1})-w^{D,N}_{\nu,m}(t_{m+1})\|^2\\
&\quad\;+2\int_{t_m}^t\Big\langle w^{S,N}_{\nu^{\epsilon},m}(s)-w^{S,N}_{\nu,m}(s),-\bfi P^N(w^{S,N}_{\nu^{\epsilon},m}(s)\nu^{\epsilon}(s)-w^{S,N}_{\nu,m}(s)\nu(s))\Big\rangle\rmd s\\
&\leq \|w^{N}_{\nu^{\epsilon},m}-w^N_{\nu,m}\|^2+C\tau_m\|w^{N}_{\nu^{\epsilon},m}-w^N_{\nu,m}\|^2\\
&\quad\;+2\int_{t_m}^t\Big\langle w^{S,N}_{\nu^{\epsilon},m}(s)-w^{S,N}_{\nu,m}(s),-\bfi P^Nw^{S,N}_{\nu^{\epsilon},m}(s)(\nu^{\epsilon}(s)-\nu(s))\Big\rangle\rmd s,
\end{align*}
which together with the iteration yields that for $t\in T_m,$ 
\begin{align*}
&\quad \;\|w^{S,N}_{\nu^{\epsilon},m}(t)-w^{S,N}_{\nu,m}(t)\|^2\\
&\leq C\sup_{t\in[0,t_m]}\|w^N_{\nu^{\epsilon},\underline{t}}-w^N_{\nu,\underline{t}}\|^2+2\int_{0}^t\Big\langle w^{S,N}_{\nu^{\epsilon},\underline{s}}(s)-w^{S,N}_{\nu,\underline{s}}(s),-\bfi P^Nw^{S,N}_{\nu^{\epsilon},\underline{s}}(s)(\nu^{\epsilon}(s)-\nu(s))\Big\rangle\rmd s.
\end{align*}

Denote $\psi_{\epsilon}(t):=\int_0^t-\bfi P^Nw^{S,N}_{\nu^{\epsilon},\underline{s}}(s)(\nu^{\epsilon}-\nu)(s)\rmd s$.
Applying the integration by parts formula and combining Proposition \ref{propwregu} give that
\begin{align*}
&\quad\int_{0}^t\Big\langle w^{S,N}_{\nu^{\epsilon},\underline{s}}(s)-w^{S,N}_{\nu,\underline{s}}(s),-\bfi P^Nw^{S,N}_{\nu^{\epsilon},\underline{s}}(s)(\nu^{\epsilon}(s)-\nu(s))\Big\rangle\rmd s\\
&=\langle w^{S,N}_{\nu^{\epsilon},\underline{t}}(t)-w^{S,N}_{\nu,\underline{t}}(t),\psi_{\epsilon}(t)\rangle -\int_0^t\Big\langle -\bfi P^Nw^{S,N}_{\nu^{\epsilon},\underline{s}}(s)\nu^{\epsilon}(s)+\bfi P^Nw^{S,N}_{\nu,\underline{s}}(s)\nu(s),\psi_{\epsilon}(s)\Big\rangle \rmd s\\
&\leq \frac{1}{4}\|w^{S,N}_{\nu^{\epsilon},\underline{t}}(t)-w^{S,N}_{\nu,\underline{t}}(t)\|^2+C\sup_{s\in[0,t]}\|\psi_{\epsilon}(s)\|^2+C\sup_{s\in[0,t]}\|\psi_{\epsilon}(s)\|.
\end{align*}
Hence, $$\|w^{S,N}_{\nu^{\epsilon},m}(t)-w^{S,N}_{\nu,m}(t)\|^2\leq  C\sup_{t\in[0,t_m]}\|w^N_{\nu^{\epsilon},\underline{t}}-w^N_{\nu,\underline{t}}\|^2+C\sup_{s\in[0,T]}\|\psi_{\epsilon}(s)\|^2+C\sup_{s\in[0,T]}\|\psi_{\epsilon}(s)\|.$$

We use the induction method to prove the compactness. Suppose that $\sup_{t\in[0,t_m]}\|w^N_{\nu^{\epsilon},\underline{t}}-w^N_{\nu,\underline{t}}\|^2\to0$ as $\epsilon\to0,$ then we show that 
$\sup_{t\in[0,t_{m+1}]}\|w^N_{\nu^{\epsilon},\underline{t}}-w^N_{\nu,\underline{t}}\|^2\to 0$ as $\epsilon\to0.$ 
Then it suffices  to show that  
\begin{align}\label{psito0}
\sup_{s\in[0,T]}\|\psi_{\epsilon}(s)\|^2+\sup_{s\in[0,T]}\|\psi_{\epsilon}(s)\|\to 0\text{ as }\epsilon\to0.
\end{align}
In fact, for $h\in \mathcal{S}_M,$ it follows from Proposition \ref{propwregu} that
\begin{align*}
\int_0^T\| \bfi Q \overline{w^{S,N}_{\nu^{\epsilon}\underline{s}}}(s)P^Nh(s)\|_{Q^{\frac{1}{2}}\bbH}^2\rmd s&\leq C\int_0^T\|Q^{\frac12}\|^2_{\mathcal{L}(H)}\|w^{S,N}_{\nu^{\epsilon}\underline{s}}(s)\|^2_{\bbH^1}\|h(s)\|^2_{H}\rmd s\\
&\leq C\int_0^T\|h(s)\|_{H_0}^2\rmd s<\infty,
\end{align*}
which together with $\nu^{\epsilon}\to \nu$ in $\mathcal{S}_M$ yields 
\begin{align*}
\lim_{\epsilon\to0}\int_0^T\big\langle -\bfi P^N w^{S,N}_{\nu^{\epsilon},\underline{s}}(s)(\nu^{\epsilon}(s)-\nu(s)),h(s)\big\rangle\rmd s=0.
\end{align*}
This means that $-\bfi P^Nw^{S,N}_{\nu^{\epsilon},\underline{\cdot}}(\nu^{\epsilon}-\nu)(\cdot)$ converges to $0$ as $\epsilon\to0$ in $L^2([0,T];\bbH)$ with respect to the weak topology. Moreover, one can show that the set $\{\psi_{\epsilon}\}_{\epsilon\in(0,1)}$ is  
a compact subset in $\mathcal{C}([0,T];\bbH_N)$ by the Ascoli theorem (see \cite[Theorem $47.1$]{topology}). In fact,
the equicontinuous of $\{\psi_{\epsilon}\}_{\epsilon\in(0,1)}$ can be deduced by
\begin{align*}
&\quad \Big\|\int_{t_1}^{t_2}-\bfi P^Nw^{S,N}_{\nu^{\epsilon},\underline{s}}(s)(\nu^{\epsilon}-\nu)(s)\rmd s\Big\|\\
&\leq \sup_{s\in[0,T]}\|w^{S,N}_{\nu^{\epsilon},\underline{s}}(s)\|\sqrt{|t_2-t_1|}\Big(\int_{t_1}^{t_2}\|\nu^{\epsilon}(s)-\nu(s)\|^2_{H^1}\rmd s\Big)^{\frac{1}{2}}\\
&\leq \sup_{s\in[0,T]}\|w^{S,N}_{\nu^{\epsilon},\underline{s}}(s)\|\sqrt{|t_2-t_1|}\Big(\int_{t_1}^{t_2}\|\nu^{\epsilon}(s)\|^2_{H^1}+\|\nu(s)\|^2_{H^1}\rmd s\Big)^{\frac{1}{2}}
\leq C\sqrt{|t_2-t_1|M}.
\end{align*}
 Since 
$$
\sup_{\epsilon\in(0,1)}\Big\|\int_0^t-\bfi P^Nw^{S,N}_{\nu^{\epsilon},\underline{s}}(s)(\nu^{\epsilon}-\nu)(s)\rmd s\Big\|_{\bbH^1}\leq C\sup_{s\in[0,T]}\|w^{S,N}_{\nu^{\epsilon},\underline{s}}(s)\|_{\bbH^1}\sup_{\epsilon\in(0,1)}\int_0^T\|(\nu^{\epsilon}-\nu)(s)\|_{H^1}\rmd s\leq C,
$$
 the compact Sobolev embedding $\bbH^1\hookrightarrow \bbH$ implies that
$\{\psi_{\epsilon}(t)\}_{\epsilon\in(0,1)}$ is compact in $\bbH$ for each fixed $t\ge 0.$ Thus
$\{\psi_{\epsilon}\}_{\epsilon\in(0,1)}$ is compact in $\mathcal{C}([0,T];\bbH_N)$,
 which combining \cite[Proposition $3.3$, Section $\mathrm{VI}$]{Functional} shows that $\psi_{\epsilon}\to 0$ in $\mathcal{C}([0,T];\bbH_N).$ Thus \eqref{psito0} is proved.

Note that \eqref{psito0} also implies that \begin{align*}
\sup_{t\in[0,t_{m+1}]}\|w^{S,N}_{\nu^{\epsilon},m}(t)-w^{S,N}_{\nu,m}(t)\|^2
\to 0 \text{ as }\epsilon\to0\quad a.s.,
\end{align*} holds for the case of $m=0$. 
Combining the induction hypothesis,   we finally obtain
\begin{align*}
\sup_{t\in[0,t_{m+1}]}\|w^{S,N}_{\nu^{\epsilon},m}(t)-w^{S,N}_{\nu,m}(t)\|^2
\to 0 \text{ as }\epsilon\to0\quad a.s.,
\end{align*}
and thus $\sup_{t\in[0,T]}\|w^{S,N}_{\nu^{\epsilon},\underline{t}}(t)-w^{S,N}_{\nu,\underline{t}}(t)\|^2\to0$ as $\epsilon\to0\;\,a.s$. The proof is finished. 
\end{proof}

The following proposition shows that the solution of the stochastic controlled equation  \eqref{small_noise} converges to that of the skeleton equation \eqref{deterministic} in distribution in $\mathcal{C}([0,T];\bbH_N)$ under certain conditions. 
 
\begin{prop}\label{laplace2}
Let $M>0,$  Assumptions \ref{refined1} and \ref{refined2} hold, and let 
$\{\nu^{\epsilon}\}_{\epsilon\in(0,1)}\subset\mathcal{P}_M$  satisfy that $\nu^{\epsilon}\xrightarrow[\epsilon\to0]{d}\nu$ as $\mathcal{S}_M$-valued random variables. Then
$u^{S,N,\epsilon}_{\nu^{\epsilon},\underline{\cdot}}(\cdot)\xrightarrow[\epsilon\to0]{d}w^{S,N}_{\nu,\underline{\cdot}}(\cdot)$ in $\mathcal{C}([0,T];\bbH_N)$. 
\end{prop}
\begin{proof}The proof is split into two steps. 

\textit{Step $1$: Show that $\{u^{S,N,\epsilon}_{\nu^{\epsilon},\underline{\cdot}}(\cdot)\}_{\epsilon\in(0,1)}$ is weakly relatively compact in $\mathcal{C}([0,T];\bbH_N)$.}

Following from  \cite[Theorem $8.6$, Chapter $3$]{Thomas86},  it suffices to prove that
\begin{itemize}
\item[(\romannumeral1)]  
$\{u^{S,N,\epsilon}_{\nu^{\epsilon},\underline{t}}(t)\}_{\epsilon\in(0,1)}$ is tight for every $t\in[0,T].$
\item[(\romannumeral2)]
There exists a family $\{\gamma_{\epsilon}(\theta,T):\theta,\epsilon\in(0,1)\}$ of nonnegative random variables satisfying
\begin{align*}
\mathbb{E}\Big[1\wedge\|u^{S,N,\epsilon}_{\nu^{\epsilon},\underline{t+\eta_1}}(t+\eta_1),u^{S,N,\epsilon}_{\nu^{\epsilon},\underline{t}}(t)\|^2\Big|\mathcal{F}_t\Big]\Big[1\wedge\|u^{S,N,\epsilon}_{\nu^{\epsilon},\underline{t}}(t),u^{S,N,\epsilon}_{\nu^{\epsilon},\underline{t-\eta_2}}(t-\eta_2)\|^2\Big]\leq \mathbb{E}[\gamma_{\epsilon}(\theta,T)|\mathcal{F}_t]
\end{align*}
for $0\leq t\leq T,0\leq \eta_1\leq \theta,$ and $0\leq \eta_2\leq \theta\wedge t$; In addition, 
\begin{align}\label{addition1}
\lim_{\theta\to 0}\sup_{\epsilon\in(0,1)}\mathbb{E}[\gamma_{\epsilon}(\theta,T)]=0
\end{align}
and 
\begin{align}\label{addition2}
\lim_{\theta\to 0}\sup_{\epsilon\in(0,1)}\mathbb{E}\big[\|u^{S,N,\epsilon}_{\nu^{\epsilon},\underline{\theta}}(\theta),u^{S,N,\epsilon}_{\nu^{\epsilon},\underline{0}}(0)\|^2\big]=0.
\end{align}
\end{itemize}

For the proof of (\romannumeral1), for arbitrary $\rho>0$ and $t\in [0,T],$ let $\Gamma_{\rho,t}:=\big\{x\in \bbH_N:\|x\|_{\bbH^1}\leq R(\rho)\big\}$ with $R(\rho)$ being determined later. The compact Sobolev embedding $\bbH^1\hookrightarrow \bbH$ implies that $\Gamma_{\rho,t}$ is compact in $\bbH.$ Since the Chebyshev inequality and \eqref{ldp_regu_2} give that
\begin{align*}
&\mathbb{P}\Big(u^{S,N,\epsilon}_{\nu^{\epsilon},\underline{t}}(t)\in \Gamma_{\rho,t}\Big)=\mathbb{P}\Big(\|u^{S,N,\epsilon}_{\nu^{\epsilon},\underline{t}}(t)\|_{\bbH^1}\leq R(\rho)\Big)\ge 1-\frac{\sup_{\epsilon\in(0,1)}\sup_{t\in[0,T]}\mathbb{E}[\|u^{S,N,\epsilon}_{\nu^{\epsilon},\underline{t}}(t)\|_{\bbH^1}]}{R(\rho)}\\
&\ge 1-\frac{C}{R(\rho)}=:1-\rho
\end{align*}
with $R(\rho)=\frac{C}{\rho},$ we obtain $\inf_{\epsilon\in(0,1)}\mathbb{P}\Big(u^{S,N,\epsilon}_{\nu^{\epsilon},\underline{t}}(t)\in \Gamma_{\rho,t}\Big)\ge 1-\rho.$ Hence, $\{u^{S,N,\epsilon}_{\nu^{\epsilon},\underline{t}}(t)\}_{\epsilon\in(0,1)}$ is tight.

For the proof of   (\romannumeral2),
 noting that $a_1a_2\leq a_1\mathbb{I}_{A}+ a_2\mathbb{I}_{A^c}$ for $0<a_1,a_2\leq 1$ and a measurable set $A$, where $\mathbb{I}$ is the indicator function,  we first prove the existence of  $\{\gamma_{\epsilon}(\theta,T)\}_{\epsilon\in(0,1)}$  such that
\begin{align*}
&\quad\;\mathbb{E}\Big[\|u^{S,N,\epsilon}_{\nu^{\epsilon},\underline{t}}(t+\eta_1)-u^{S,N,\epsilon}_{\nu,\underline{t}}(t)\|^2\mathbb{I}_{\{\bar{t}-t\ge \frac{1}{2}\tau_{min}\}}+\|u^{S,N,\epsilon}_{\nu^{\epsilon},\underline{t}}(t)-u^{S,N,\epsilon}_{\nu,\underline{t}}(t-\eta_2)\|^2\mathbb{I}_{\{\bar{t}-t<\frac{1}{2}\tau_{min}\}}\Big|\mathcal{F}_t\Big]\\
&\leq \mathbb{E}[\gamma_{\epsilon}(\theta,T)|\mathcal{F}_t]
\end{align*}
  for $0\leq t\leq T,\;0\leq \eta_1\leq \theta< (1\wedge \frac{1}{2}\tau_{min})$ and $0\leq \eta_2\leq t\wedge \theta$, where $\bar{t}:=\min\{t_m:t_m\ge t\}.$ Recall that $\tau_{min}$ is given in Remark \ref{rmktau_min}. 
  
Note that for $t\in\{t\in[0,T]:\bar{t}-t\ge \frac{1}{2}\tau_{min}\},$
  \begin{align*}
  &\quad \;\|u^{S,N,\epsilon}_{\nu^{\epsilon},\underline{t}}(t+\eta_1)-u^{S,N,\epsilon}_{\nu^{\epsilon},\underline{t}}(t)\|^2\\
  &=\Big\|\int_t^{t+\eta_1}-\frac{\epsilon}{2}P^NF_Qu^{S,N,\epsilon}_{\nu^{\epsilon},\underline{t}}(s)-\bfi P^Nu^{S,N,\epsilon}_{\nu^{\epsilon},\underline{t}}(s)\nu^{\epsilon}(s)\rmd s-\int_t^{t+\eta_1}\bfi\sqrt{\epsilon}P^Nu^{S,N,\epsilon}_{\nu^{\epsilon},\underline{t}}(s)\rmd W(s)\Big\|^2\\
  &\leq C\theta
  \int_t^{t+\theta}\|u^{S,N,\epsilon}_{\nu^{\epsilon},\underline{t}}(s)\|^2(\epsilon^2\|F_Q\|^2_{L^{\infty}(\OO)}+\|\nu^{\epsilon}(s)\|_{H^1}^2)\rmd s+C\epsilon\sup_{\eta_1\leq \theta}\Big\|\int_t^{t+\eta_1}P^Nu^{S,N,\epsilon}_{\nu^{\epsilon},\underline{t}}(s)\rmd W(s)\Big\|^2\\
  &=:\mathcal{I}_1,
  \end{align*}
  and for $t\in\{t\in[0,T]:\bar{t}-t<\frac{1}{2}\tau_{min}\},$
  \begin{align*}
  &\quad\; \|u^{S,N,\epsilon}_{\nu^{\epsilon},\underline{t}}(t)-u^{S,N,\epsilon}_{\nu^{\epsilon},\underline{t}}(t-\eta_2)\|^2\\
  &\leq C\theta \int_{t-\theta}^t\|u^{S,N,\epsilon}_{\nu^{\epsilon},\underline{t}}(s)\|^2(\epsilon^2\|F_Q\|_{L^{\infty}(\OO)}^2+\|\nu^{\epsilon}(s)\|_{H^1}^2)\rmd s+C\epsilon\sup_{\eta_2\leq t\wedge \theta}\Big\|\int_{t-\eta_2}^tP^N u^{S,N,\epsilon}_{\nu^{\epsilon},\underline{t}}(s)\rmd W(s)\Big\|^2\\
  &=:\mathcal{I}_2.
  \end{align*}
 The random variable $\gamma_{\epsilon}(\theta, T)$ is chosen as $\gamma_{\epsilon}(\theta,T)=\mathcal{I}_1\mathbb{I}_{\{\bar{t}-t\ge \frac{1}{2}\tau_{min}\}}+\mathcal{I}_2\mathbb{I}_{\{\bar{t}-t< \frac{1}{2}\tau_{min}\}}$ for each $\epsilon\in(0,1)$ and $\theta<(1\wedge \frac12\tau_{min}).$ And
we remark that if $\frac12\tau_{min}\leq\theta<1,$ then we let $\gamma_{\epsilon}(\theta,T)\equiv 1$.
 Then it follows from   the Burkholder--Davis--Gundy inequality and \eqref{ldp_regu_1} that  \begin{align*}
 &\quad\;\sup_{\epsilon\in(0,1)}\mathbb{E}[\gamma_{\epsilon}(\theta,T)]\\
 &\leq \sup_{\epsilon\in(0,1)}\Big\{C\theta^2\epsilon^2+C\theta+ C\epsilon  \mathbb{E}\Big[\int_t^{t+\theta}\|u^{S,N,\epsilon}_{\nu^{\epsilon},\underline{t}}(s)\|^2\|Q^{\frac{1}{2}}\|^2_{\mathcal{L}^1_2}\rmd s\Big]+C\epsilon \mathbb{E}\Big[\int_{t-\theta}^t\|u^{S,N,\epsilon}_{\nu^{\epsilon},\underline{t}}(s)\|^2\|Q^{\frac{1}{2}}\|^2_{\mathcal{L}^1_2}\rmd s\Big]\Big\}\\
 &\leq \sup_{\epsilon\in(0,1)}\Big[C\theta(\theta\epsilon ^2+1)+C\epsilon\theta\Big]\leq C\theta\to0\text{ as }\theta\to0,
  \end{align*}
 which proves  \eqref{addition1}.
   Finally,    it is deduced from 
\begin{align*}
\lim_{\theta\to0}\sup_{\epsilon\in(0,1)}
\mathbb{E}\Big[\|u^{S,N,\epsilon}_{\nu^{\epsilon},0}(\theta)-u^{S,N,\epsilon}_{\nu^{\epsilon},0}(0)\|^2\Big]\leq \lim_{\theta\to0}\sup_{\epsilon\in(0,1)}\Big[C\theta(\theta\epsilon ^2+1)+C\epsilon \theta\Big]=0
\end{align*}
that \eqref{addition2} 
is satisfied, which 
 finishes the proof that $\{u^{S,N,\epsilon}_{\nu^{\epsilon},\underline{\cdot}}(\cdot)\}_{\epsilon\in(0,1)}$ is weakly relatively compact in $\mathcal{C}([0,T];\bbH_N).$ 

\textit{Step $2$: Show that $u^{S,N,\epsilon}_{\nu^{\epsilon},\underline{\cdot}}(\cdot)\xrightarrow[\epsilon\to0]{d}w^{S,N}_{\nu,\underline{\cdot}}(\cdot)$ if $\nu^{\epsilon}\xrightarrow[\epsilon\to0]{d}\nu.$}

 Since $\{\nu^{\epsilon}\}$ is tight and $\mathcal{S}_M$ is a compact Polish space, $\{\nu^{\epsilon}\}$ is weakly relatively compact based on the Prohorov theorem (see e.g. \cite[Theorem $A.3.15$]{Weakbook}). Thus 
$\{(u^{S,N,\epsilon}_{\nu^{\epsilon},\underline{\cdot}}(\cdot),\nu^{\epsilon})\}_{\epsilon\in(0,1)}$ is weakly relatively compact in $\mathcal{C}([0,T];\bbH_N)\times \mathcal{S}_M$. Hence, there exists a subsequence $\epsilon_n\to0$ $ (\text{as }n\to\infty)$ such that $\{(u^{S,N,\epsilon_n}_{\nu^{\epsilon_n},\underline{\cdot}}(\cdot),\nu^{\epsilon_n})\}_{\epsilon_n\in(0,1)}$ converges in distribution to an element taking values in $\mathcal{C}([0,T];\bbH_N)\times \mathcal{S}_M.$ It follows from the Skorohod representation theorem (see e.g. \cite[Theorem $A.3.9$]{Weakbook}) that there exists a probability space $(\tilde{\Omega},\tilde{\mathcal{F}},\tilde{\mathbb{P}})$ on which a $\mathcal{C}([0,T];\bbH_N)\times \mathcal{S}_M$-valued random variable $(\widetilde{u^{S,N}_{\underline{\cdot}}}(\cdot),\tilde{\nu})$ is such that $\{(u^{S,N,\epsilon_n}_{\nu^{\epsilon_n},\underline{\cdot}}(\cdot),\nu^{\epsilon_n})\}_{\epsilon_n\in(0,1)}$ converges to $(\widetilde{u^{S,N}_{\underline{\cdot}}}(\cdot),\tilde{\nu})$ in distribution. Denote by $\mathbb{E}_{\tilde{\mathbb{P}}}$ the expectation with respect to $\tilde{\mathbb{P}}.$ We need to show that $\widetilde{u^{S,N}_{\underline{\cdot}}}(\cdot)$ satisfies that for $t\in T_m,$
\begin{equation}
\left\{
\begin{array}{ll}
\rmd \widetilde{u^{D,N}_{m}}(t)=\bfi \Delta \widetilde{u^{D,N}_{m}}(t)\rmd t+\bfi \lambda S^N(t-t_m)P^N|\widetilde{u^{N}_{m}}|^2\widetilde{u^{N}_{m}}\rmd t,\quad  \widetilde{u^{D,N}_{m}}(t_m)=\widetilde{u^{S,N}_{m-1}}(t_m),\vspace{1mm}\\
\rmd \widetilde{u^{S,N}_{m}}(t)=-\bfi P^N\widetilde{u^{S,N}_{m}}(t)\tilde{\nu}(t)\rmd t,\quad  \widetilde{u^{S,N}_{m}}(t_m)=\widetilde{u^{D,N}_{m}}(t_{m+1}).
\end{array}
\right.
\end{equation}

To this end, for $t\in T_m,$ define the map $\Upsilon_t:\mathcal{C}([0,T];\bbH_N)\times \mathcal{S}_M\to[0,1]$ by
\begin{align*}
\Upsilon_t(f,\phi)=1\wedge \Big\|f(t)-S^N(\tau_m)\big(f(t_m)+\bfi\lambda P^N|f(t_m)|^2f(t_m)\tau_m\big)+\int_{t_m}^t\bfi P^Nf(s)\phi(s)\rmd s\Big\|.
\end{align*}
We claim that $\Upsilon_t$ is continuous and bounded.
In fact, noting that $\mathcal{C}([0,T];\bbH^1)$ is dense in $\mathcal{C}([0,T];\bbH),$ we let $f_n\to f$ in $\mathcal{C}([0,T];\bbH)$ with $\sup_{n\in\mathbb{N}_+}\|f_n\|_{\bbH^1}\vee\|f\|_{\bbH^1}<\infty$ and let $\phi_n\to\phi$ in $\mathcal{S}_M$ with respect to the weak topology. 
By $\big|1\wedge \|x_1\|-1\wedge \|x_2\|\big|\leq 1\wedge \|x_1-x_2\|\leq \|x_1-x_2\|$, we arrive at
\begin{align}\label{conver}
&\quad|\Upsilon_t(f_n,\phi_n)-\Upsilon_t(f,\phi)|\notag\\&\leq C\|f_n-f\|_{\mathcal{C}([0,T];\bbH)}\Big(1+\int_{t_m}^t\|\phi_n(s)\|_{H^1}\rmd s\Big)+\Big\|\int_{t_m}^t\bfi P^Nf(s)(\phi_n(s)-\phi(s))\rmd s\Big\|.
\end{align}
Similar to the proof of the convergence of $\{\psi_{\epsilon}\}_{\epsilon\in(0,1)}$ in Proposition \ref{laplace1}, the last term in the right hand of \eqref{conver} converges to $0$ uniformly with respect to $t$.

Hence, 
\begin{align*}
\lim_{n\to\infty}\mathbb{E}\Big[\Upsilon_t(u^{S,N,\epsilon_n}_{\nu^{\epsilon_n},\underline{\cdot}}(\cdot),\nu^{\epsilon_n})\Big]=\mathbb{E}_{\tilde{\mathbb{P}}}[\Upsilon_t(\widetilde{u^{S,N}_{\underline{\cdot}}}(\cdot),\tilde{\nu})];
\end{align*}
see e.g. \cite[Page $375$, Appendix $A.3$]{Weakbook}.
Since for $t\in T_m,$
\begin{align*}
\mathbb{E}\Big[\Upsilon_t(u^{S,N,\epsilon_n}_{\nu^{\epsilon_n},\underline{\cdot}}(\cdot),\nu^{\epsilon_n})\Big]&=1\wedge \mathbb{E}\Big[\Big\|\epsilon_n\int_{t_m}^t\frac12P^NF_Qu^{S,N,\epsilon_n}_{\nu^{\epsilon_n},m}(s)\rmd s+\sqrt{\epsilon_n}\int_{t_m}^t\bfi P^Nu^{S,N,\epsilon_n}_{\nu^{\epsilon_n},m}(s)\rmd W(s)\Big\|\Big]\\
&\leq  \frac{\epsilon_n}{2}\int_{0}^T\|F_Q\|_{L^{\infty}(\OO)}\|u^{S,N,\epsilon_n}_{\nu^{\epsilon_n},m}(s)\|\rmd s+\sqrt{\epsilon_n}\Big(\mathbb{E}\Big\|\int_{t_m}^t P^Nu^{S,N,\epsilon_n}_{\nu^{\epsilon_n},m}(s)\rmd W(s)\Big\|^2\Big)^{\frac{1}{2}}\\
&\leq C\epsilon_n\to0\text{ as }n\to\infty,
\end{align*}
we obtain
$
\mathbb{E}_{\tilde{\mathbb{P}}}\Big[\Upsilon_t(\widetilde{u^{S,N}_{\underline{\cdot}}}(\cdot),\tilde{\nu})
\Big]=0.$
 It follows from  the definition of $\Upsilon_t$ that 
\begin{align*}
\widetilde{u^{S,N}_{\underline{\cdot}}}(\cdot)=\mathcal{G}^0\Big(\int_0^{\cdot}\tilde{\nu}(s)\rmd s\Big)\quad \tilde{\mathbb{P}}\text{-a.s.}
\end{align*}
Moreover, due to $(u^{S,N,\epsilon_n}_{\nu^{\epsilon_n},\underline{\cdot}}(\cdot),\nu^{\epsilon_n})\xrightarrow[\epsilon_n\to0]{d}(\widetilde{u^{S,N}_{\underline{\cdot}}}(\cdot),\tilde{\nu})$, we have
$\nu^{\epsilon_n}\xrightarrow[\epsilon_n\to0]{d}\tilde{\nu}$, which together with $\nu^{\epsilon}\xrightarrow[\epsilon\to0]{d}\nu$ yields that $\nu\overset{d}{=}\tilde{\nu}$ and consequently $w^{S,N}_{\nu,\underline{\cdot}}(\cdot)\overset{d}{=}\widetilde{u^{S,N}_{\underline{\cdot}}}(\cdot).$ Therefore, $$(u^{S,N,\epsilon_n}_{\nu^{\epsilon_n},\underline{\cdot}}(\cdot),\nu^{\epsilon_n})\xrightarrow[\epsilon_n\to0]{d}(w^{S,N}_{\nu,\underline{\cdot}}(\cdot),\nu).$$ Repeating the above procedure, we derive that for any subsequence $\vartheta_n\to0,$ there exists some subsubsequence $\vartheta_{n_k}\to0,$ such that  $(u^{S,N,\vartheta_{n_k}}_{\nu^{\vartheta_{n_k}},\underline{\cdot}}(\cdot),\nu^{\vartheta_{n_k}})\xrightarrow[\vartheta_{n_k}\to0]{d}(w^{S,N}_{\nu,\underline{\cdot}}(\cdot),\nu),$ which finally implies that
$(u^{S,N,\epsilon}_{\nu^{\epsilon},\underline{\cdot}}(\cdot),\nu^{\epsilon})\xrightarrow[\epsilon\to0]{d}(w^{S,N}_{\nu,\underline{\cdot}}(\cdot),\nu);$ see e.g. \cite[Theorem $2.6$]{convergence}. 

Combining \textit{Steps} $1$-$2$, we finish the proof.
\end{proof}

\begin{proof}[Proof of Theorem \ref{mainLDP}]
Following  \cite[Theorem $4.4$]{BD1} or \cite[Theorem $5$]{BD2}, it suffices to prove that
\begin{itemize}
\item[(\romannumeral 1)] for any fixed $M<\infty$, 
\begin{align*}
K_M=\Big\{\mathcal{G}^0\Big(\int_0^{\cdot}\nu (s)\mathrm{d}s\Big),\nu\in\mathcal{S}_M\Big\}
\end{align*}
is a compact subset of $\mathcal{C}([0,T];\bbH_N)$;
\item[(\romannumeral 2)] for $M<\infty$ and $\{\nu^{\epsilon}\}_{\epsilon\in(0,1)}\subset \mathcal{P}_M$ such that $\nu^{\epsilon}\xrightarrow[\epsilon\to0]{d}\nu$ as $\mathcal{S}_M$-valued random variables,
\begin{align*}
\mathcal{G}^{\epsilon}\Big(\sqrt{\epsilon}W+\int_0^{\cdot}\nu^{\epsilon}(s)\rmd s\Big)\xrightarrow[\epsilon\to0]{d} \mathcal{G}^0\Big(\int_0^{\cdot}\nu (s)\rmd s\Big),
\end{align*}
\end{itemize}
which are given in Propositions \ref{laplace1} and \ref{laplace2}, respectively.  
\end{proof}

 Recall that the mass conservation law $\|u(t)\|^2=\|u_0\|^2\;\forall \,t\in[0,T]$ holds for both the stochastic NLS equation  \eqref{schrod} and the split equation \eqref{u^D_tau_m(t)}.
Even though the mass can not be preserved exactly by the adaptive fully discrete scheme \eqref{Itoform},  the error of the masses between   solutions of \eqref{Itoform}   and \eqref{schrod} can be given by means of the LDP for the numerical solution.

\begin{coro}
 Under assumptions in Theorem \ref{mainLDP},  for any $\rho>0$, there is  some $\epsilon_0>0$ such that for $ \epsilon<\epsilon_0,$
\begin{align*}
&\quad\exp\Big\{-\frac{1}{\epsilon}\inf_{x\in G^1_{\rho}}I(x)\Big\}+\exp\Big\{-\frac{1}{\epsilon}\inf_{x\in G^2_{\rho}}I(x)\Big\}\\&\leq\mathbb{P}\Big(\Big|\sup_{t\in[0,T]}\|{u^{S,N,\epsilon}_{t,\underline{t}}}\|^2-\|u_0\|^2\Big|\ge\rho\Big)
\leq \exp\Big\{-\frac{1}{\epsilon}\inf_{x\in F^1_{\rho}}I(x)\Big\}+\exp\Big\{-\frac{1}{\epsilon}\inf_{x\in F^2_{\rho}}I(x)\Big\},
\end{align*}
where $G^1_{\rho}=\big\{x\in \mathcal{C}([0,T];\bbH_N):\sup_{t\in[0,T]}\|x(t)\|^2>\|u_0\|^2+\rho+\hat{\varepsilon}\big\},\;G^2_{\rho}=\big\{x\in \mathcal{C}([0,T];\bbH_N):\sup_{t\in[0,T]}\|x(t)\|^2<\|u_0\|^2-\rho-\hat{\varepsilon}\big\}$ with $\hat{\varepsilon}>0$ being a small number, $F^1_{\rho}=\big\{x\in \mathcal{C}([0,T];\bbH_N):\sup_{t\in[0,T]}\|x\|^2\ge \|u_0\|^2+\rho\big\},\; F^2_{\rho}=\big\{x\in \mathcal{C}([0,T];\bbH_N):\sup_{t\in[0,T]}\|x(t)\|^2\leq \|u_0\|^2-\rho\big\}$, and $I$ is given in Theorem \ref{mainLDP}.
\end{coro}
\begin{proof}
It is straightforward that 
\begin{align*}
&\quad\mathbb{P}\Big(\Big|\sup_{t\in[0,T]}\|{u^{S,N,\epsilon}_{t,\underline{t}}}\|^2-\|u_0\|^2\Big|\ge\rho\Big)\\
&= \mathbb{P}\Big(\sup_{t\in[0,T]}\|{u^{S,N,\epsilon}_{t,\underline{t}}}\|^2\ge \|u_0\|^2+\rho\Big)+\mathbb{P}\Big(\sup_{t\in[0,T]}\|{u^{S,N,\epsilon}_{t,\underline{t}}}\|^2\leq \|u_0\|^2-\rho\Big)\\
&=:\mathcal{I}\mathcal{I}_1+\mathcal{I}\mathcal{I}_2.
\end{align*}
Note that $\{\omega: u^{S,N,\epsilon}_{\cdot,\underline{\cdot}}(\omega)\in G^1_{\rho}\}\subset \big\{\omega:\sup_{t\in[0,T]}\|{u^{S,N,\epsilon}_{t,\underline{t}}}(\omega)\|^2\ge \|u_0\|^2+\rho\big\}$ and $\{\omega: u^{S,N,\epsilon}_{\cdot,\underline{\cdot}}(\omega)\in G^2_{\rho}\}\subset \big\{\omega:\sup_{t\in[0,T]}\|{u^{S,N,\epsilon}_{t,\underline{t}}}(\omega)\|^2\leq \|u_0\|^2-\rho\big\}.$ Terms $\mathcal{I}\mathcal{I}_j$ can be estimated by the LDP upper bound (resp. the LDP lower bound) with the closed subset $F^j_{\rho}$ (resp. the open subset $G^j_{\rho}$) for $j=1,2$. \end{proof}

\bibliographystyle{plain}
\bibliography{adaptive_schr.bib}

\end{document}